\documentclass[12pt]{amsart}

\usepackage{amssymb}
\usepackage{amsfonts}
\usepackage{amscd}
\usepackage[T1]{fontenc}

\usepackage[mathlines]{lineno}

\usepackage{color}

\makeatletter
\let\@wraptoccontribs\wraptoccontribs
\@namedef{subjclassname@2010}{%
  \textup{2010} Mathematics Subject Classification}
\makeatother

\newcommand{\abar}{\bar{a}}
\newcommand{\bbar}{\bar{b}}

\newcommand{\xtil}{\tilde{x}}
\newcommand{\xbar}{\bar{x}}

\newcommand{\ytil}{\tilde{y}}
\newcommand{\ybar}{\bar{y}}

\newcommand{\alphabar}{\bar{\alpha}}
\newcommand{\betabar}{\bar{\beta}}
\newcommand{\etabar}{\bar{\eta}}
\newcommand{\thetabar}{\bar{\theta}}

\newcommand{\inject}{\hookrightarrow}
\newcommand{\surject}{\twoheadrightarrow}

\newcommand{\discr}{\operatorname{discr}}

\newcommand\dcl{\operatorname{dcl}}

\def\FF{{\mathbb F}}

\def\NN{{\mathbb N}}

\def\PP{{\mathbb P}}
\def\QQ{{\mathbb Q}}
\def\RR{{\mathbb R}}

\def\ZZ{{\mathbb Z}}

\newcommand{\Auxc}{\mathcal{S}}  % c = congruence
\newcommand{\Auxe}{\mathcal{T}}  % e = equality
\newcommand{\canc}{\mathfrak{s}}  % c = congruence
\newcommand{\cane}{\mathfrak{t}}  % e = equality

\newcommand{\Sp}{\operatorname{Sp}}

\newcommand{\Aux}{\mathcal{A}}
\newcommand{\Main}{\mathcal{M}}

\newcommand{\Loag}{L_{\mathrm{oag}}}
\newcommand{\Lint}{L_{\mathrm{qe}}}
\newcommand{\Lsyn}{L_{\mathrm{syn}}}
\newcommand{\Lpres}{L_{\mathrm{Pres}}}
\newcommand{\conv}{^{\mathrm{conv}}}

\newcommand{\bigdcup}{\mathop{\dot{\bigcup}}}
\newcommand{\dcup}{\mathrel{\dot{\cup}}}

\newcommand{\diamrel}{\mathrel{\diamond}}

\newtheorem{theorem}{Theorem}[section]
\newtheorem{lem}[theorem]{Lemma}
\newtheorem{cor}[theorem]{Corollary}
\newtheorem{prop}[theorem]{Proposition}

\theoremstyle{definition}
\newtheorem{defn}[theorem]{Definition}
\newtheorem{notation}[theorem]{Notation}
\newtheorem{rem}[theorem]{Remark}

\newtheorem{def-prop}[theorem]{Proposition-Definition}
\newtheorem{def-theorem}[theorem]{Theorem-Definition}
\newtheorem{def-lem}[theorem]{Lemma-Definition}

\theoremstyle{remark}

\let\subsetold=\subset
\renewcommand{\subset}{\mathrel{{\color{red}\underset{\rlap{Replace by subseteq!}}{\subsetold}}}}
\let\supsetold=\supset
\renewcommand{\supset}{\mathrel{{\color{red}\underset{\rlap{Replace by supseteq!}}{\supsetold}}}}

\newcommand{\ciff}{\mathrel{\mathord{\iff}\llap{$(\,\,)$}}}

\begin{document}

\title{Quantifier elimination in ordered abelian groups}

\author{Raf Cluckers}
\address{Universit\'e Lille 1, Laboratoire Painlev\'e, CNRS - UMR 8524, 
Cit\'e Scientifique, 59655
Villeneuve d'Ascq Cedex, France, and,
Katholieke Universiteit Leuven, Department of Mathematics,
Celestijnenlaan 200B, B-3001 Leu\-ven, Bel\-gium}
\email{Raf.Cluckers@math.univ-lille1.fr}
\urladdr{http://math.univ-lille1.fr/$\sim$cluckers}

\author{Immanuel Halupczok}
\address{Institut f\"ur Mathematische Logik und Grundlagenforschung,
Universit\"at M\"unster,
Einsteinstra\ss e 62,
48149 M\"unster,
Germany}
\email{math@karimmi.de}
\urladdr{http://www.immi.karimmi.de/en.math.html}
\thanks{The second author was supported by the SFB~878 of the Deutsche Forschungsgemeinschaft}


\begin{abstract}
We give a new proof of quantifier elimination in the theory of all ordered abelian groups in a suitable language. More precisely, this is only
``quantifier elimination relative to ordered sets'' in the following sense.
Each definable set in the group is a union of a family of quantifier free definable sets,
where the parameter of the family runs over a set definable (with quantifiers) in a
sort which carries the structure of an ordered set with some additional unary
predicates.

As a corollary, we find that all definable functions in ordered abelian groups are piecewise affine linear on finitely many definable pieces.
\end{abstract}

\maketitle

%\newpage

\section*{Introduction}

Quantifier elimination is well known in some particular ordered abelian groups
like $\QQ$ and $\ZZ$. Somewhat less well known is that there also exists a
quantifier elimination result for the theory of all ordered abelian groups.
For sentences, this has already been proven by Gurevich \cite{Gur.oAGelProp} in 1964.
Later, Gurevich and Schmitt enhanced this to
treat arbitrary formulas (\cite{Gur.decProb}, \cite{Sch.habil}).
The main goal of the present article is to introduce a new
language $\Lint$ with similar kind of quantifier elimination, which is more intuitive
and hopefully more useful for applications.

As a corollary, we obtain that every definable function $f\colon G^n \to G$
in ordered abelian groups is piecewise linear, i.e., there exists a partition of $G^n$
into finitely many definable sets such that the restriction of $f$ to any of these
sets is of the form $f(x_1, \dots, x_n) = \frac{1}{s}(\sum_i r_ix_i + b)$ with $r_i, s \in \ZZ$
and $b \in G$.
This result has been proven in the special case of groups of finite regular rank
by Belegradek-Verbovskiy-Wagner \cite{BVW.ordGrp} (using a version of quantifier
elimination in this context from Weispfenning, \cite{Wei.ordAbGrp}),
but to our knowledge, it has yet not been written down in full generality before.
Our interest in this result came from valued fields. In the $p$-adics, definable maps
can piecewise be approximated by fractional polynomials; see \cite{iC.lipQp}. To
get a similar result in valued fields with arbitrary value group, one necessary
ingredient is piecewise linearity of definable maps in the value group.

Our quantifier elimination result could be deduced rather easily from the results
of Gurevich and Schmitt. However, we discovered their results
only after we had already written our own complete proof. We decided to include our proof
in this article anyway to keep it self-contained and
because both \cite{Gur.decProb} and \cite{Sch.habil} are difficult to obtain.
Moreover we are using a more modern formalism; in particular, we shall see
that this kind of quantifier elimination naturally lives in many-sorted structures,
which were not used by Gurevich and Schmitt.

From now on, we write ``oag'' for ``ordered abelian group''.

There is no really simple language in which oags have quantifier elimination; the main reason is that oags may have many
convex definable subgroups, which come in several definable families.
Parametrizing one such family with 
a suitable imaginary sort yields a uniform way to interpret an arbitrary
ordered set in an appropriate oag.
Since ordered sets have no good quantifier elimination
language, the best one can hope for in oags is
``quantifier elimination relative to ordered sets''; this is indeed we get.

Let us examine more closely what is needed in a quantifier elimination language.
Recall that in the oag $\ZZ$, we have quantifier elimination in the
Presburger language $\Lpres := \{0, 1, +, <, \equiv_m\}$ (where $a \equiv_m b$ iff $a - b \in m\ZZ$).
The same language also yields quantifier elimination in any fixed
oag without (non-trivial) convex definable subgroup; in that case,
$1$ is defined to be the minimal positive element if this exists and $1 = 0$ otherwise.
If $G$ is a fixed group with finitely many convex definable subgroups $H$, then
the quotients $G/H$ are interpretable in $G$, and to get quantifier elimination,
it is necessary (and sufficient) to have $\Lpres$ not only on $G$, but also on all those quotients.

Now let us sketch the complete quantifier elimination language $\Lint$;
it should allow for oags with infinite families of convex definable subgroups and moreover we want to work in the theory of all oags and not just in
a fixed one. To treat 
infinite families of convex definable subgroups, we will add new sorts to $\Lint$ (called ``auxiliary sorts'')
with canonical parameters for some of them; let us write $G_\alpha$ for the group corresponding to the canonical parameter $\alpha$. We will still need the Presburger language on all quotients
$G/G_\alpha$; roughly this will be formalized as follows: each quantifier free binary $\Lpres$-relation
$x \diamrel y$ (for $x, y \in G/G_\alpha$) becomes a ternary relation $\xtil \diamrel_\alpha \ytil$
(for $\xtil, \ytil \in G$ and $\alpha$ in an auxiliary sort) which holds iff the images of $\xtil$ and $\ytil$ in $G/G_\alpha$ satisfy $\diamond$. (For example, for each $m \in \NN$, we have a relation $\xtil \equiv_{m,\alpha} \ytil$ which holds iff $\xtil - \ytil \in mG + G_\alpha$.)

Apart from that, three more things are needed in the language $\Lint$. On the auxiliary sorts, we have the order relation induced by inclusion of the corresponding subgroups and some unary predicates corresponding to certain
properties of the groups $G/G_\alpha$ (which otherwise could not be expressed without quantifiers);
moreover, we will need a variant of the congruence relation $\equiv_{m,\alpha}$ introduced above.

Our main result (Theorem~\ref{thm:qe-int}) is that 
in $\Lint$, we have ``quantifier elimination relative to the auxiliary sorts'' in the following strong sense.
Every definable subset in $G$ is a union of a family of quantifier
free definable sets, parametrized by an auxiliary set. This auxiliary set is
defined by a formula which may use quantifiers, but it uses only the auxiliary part of
$\Lint$ (i.e., some ordered sets with unary predicates).

This kind of relative quantifier elimination might sound weak, but despite the fact
that ordered sets have no good quantifier elimination,
their model theory is well understood; see e.g.\ \cite{Rub.linOrd}
or \cite[ch.\ 12.f]{Poi.modTh}. (This is also true for ordered sets with unary predicates,
also called ``colored chains''.) Relative quantifier elimination allows to lift
good model theoretic properties from ordered sets to oags;
for example, Gurevich and Schmitt did this for NIP
in \cite{GS.oAGNIP}. Other results about oags may be deduced directly
from relative quantifier elimination, without any knowledge of the auxiliary sorts at all;
an example for this is our corollary about piecewise linearity of definable maps.

% .
% As an example, we prove that definable maps are piecewise linear; more precisely,
% by ``linear'', we mean linear in the sense of polynomial, i.e., of the form
% $(x_1, \dots, x_n) \mapsto \sum a_ix_i + b$,
% and ``piecsewise'' means: after partitioning the domain into finitely many definable sets.
% This result has been proven in the special case of groups of finite regular rank
% by Belegradek-Verbovskiy-Wagner \cite{BVW.ordGrp} (using a version of quantifier
% elimination in this context from Weispfenning, \cite{Wei.ordAbGrp}),
% but to our knowlegde, it has yet not been written down in full generality before.

%The main reason we were interested in quantifier elimination in ordered abelian groups
%was that we needed to understand definable maps.
%(This is needed in \cite{iC.lip00}, where we plan to approximate definable functions in
%Henselian fields by fractional monomials, in analogy to the $p$-adic case in \cite{iC.lipQp}.)

%Indeed, using quantifier elimination one can easily prove that

To prove relative quantifier elimination in $\Lint$, it is useful to
simultaneously prove it in a second language $\Lsyn$ which has certain good syntactic
properties. These allow us to reduce relative quantifier elimination
to eliminating a single existential quantifier of a formula which contains no other quantifiers, as one does it in usual quantifier elimination.
This language $\Lsyn$ is very close to the one used by Gurevich and Schmitt
in their quantifier elimination results.

The article is organized as follows.

In Section~\ref{sect:results}, we present the main results: quantifier elimination in the languages
$\Lint$ and $\Lsyn$ (Theorems~\ref{thm:qe-int} and~\ref{thm:qe-form}) and piecewise linearity of definable functions (Corollary~\ref{cor:fctn}).
We also state the general result on relative quantifier elimination 
in languages with good syntactic properties (Proposition~\ref{prop:no-groups}).
In this section, the languages are defined as quickly as possible,
postponing explanations to the next section.
At the end of the section, we
explain the relation between $\Lsyn$ and the language used by Schmitt.

In Section~\ref{sect:details}, we prove some first basic properties of the languages,
which also yields some motivation. Then we show how to translate between
$\Lsyn$ and $\Lint$, allowing us to switch freely between those languages while doing
quantifier elimination.

Section~\ref{sect:proofs} contains the main proofs.
First, we prove Proposition~\ref{prop:no-groups}. Then we do the actual elimination
of one existential quantifier; this is done in the language $\Lint$.
The whole proof is constructive, so it can be turned into an algorithm
for quantifier elimination.

Section~\ref{sect:ex} contains some examples illustrating the language
$\Lint$; in particular, they show how arbitrary ordered sets can
be interpreted in oags.

% without using the full
%quantifier elimination result. Instead, piecewise linearity is first proven
%in regular ordered abelian groups, where quantifier elimination was already
%proven in xxxxxx; this result is then extended to arbitrary ordered abelian
%groups using a principle of Feferman-Vaugt (xxxxxx) and the structure of
%$\aleph_1$-saturated ordered abelian groups (xxxxxxx).

\subsection{Acknowledgment}

We are very grateful to Fran\c{c}oise Delon for a lot of
interesting discussions and for several valuable concrete suggestions
concerning this article.

\section{The results}
\label{sect:results}

\subsection{Generalities and basic notation}

We use the convention that $0 \notin \NN$, and we write $\NN_0$ for
$\NN \cup \{0\}$ and $\PP$ for the set of primes.

In the whole article, $(G,+,<)$ will denote an ordered abelian group (``oag''), that is,
a group with a total order which is compatible with the group operation:
$a < b$ iff $a+c<b+c$ for all $a,b,c\in G$.
It is easy to see that such a group is always torsion free.
Such groups appear naturally, for example, as valuation group of
(Krull) valued fields.
An oag is called \emph{discrete}, if it has a minimal positive element
and \emph{dense} otherwise.

We write $\Loag = \{0, +, <\}$ for the language of oags
and unless stated otherwise, we always work in the theory of all oags.

For $a \in G$, we write $\langle a \rangle\conv$ for the smallest convex subgroup
of $G$ containing $a$; for $a, b \in G$ and $m \in \NN$, $a \equiv_m b$
means that $a$ and $b$ are congruent modulo $m$ in the sense that $a - b \in mG$.

\subsection{A language for quantifier elimination}
\label{sect:thm-qe}

We now give a precise definition of the quantifier elimination language $\Lint$; motivation and and additional explanations will be given in Section~\ref{sect:details}.
An introduction to $\Lint$ with much more motivation and examples can be found in \cite{i.oAGlang}.
Note that all of $\Lint$ will be $\Loag$-definable (where new sorts in $\Lint$ are considered as imaginary sorts of $\Loag$).

We start by introducing the new sorts of $\Lint$:
sorts with canonical parameters for some definable families of convex subgroups.
These new sorts will be
called \emph{auxiliary sorts}; in contrast, the sort of the ordered abelian group itself will be called the \emph{main sort}.

%As announced, we will only be able to eliminate the main sort quantifiers.
%However,
%the auxiliary sorts will essentially be colored chains: sets with
%the order relation induced by inclusion of the corresponding subgroups,
%and some predicates. For such colored chains, model theory is
%well known (see e.g.\ xxxxxxx).

%Let us now introduce the language $\Lint$ for this quantifier elimination result
%as quickly as possible. More motivation an details will be given in
%Section~\ref{sect:details}.

For each positive integer $n$, we consider three families of convex definable subgroups, parametrized by sorts which we denote by $\Auxc_n$, $\Auxe_n$, and $\Auxe^+_n$.
Although in $\Lint$ we will have these sorts only for $n$ prime, it is useful
to define them for all $n$. Examples illustrating the following definition are given in Section~\ref{sect:ex}.

\begin{defn}\label{defn:aux}
\begin{enumerate}
\item
For $n \in \NN$ and $a \in G \setminus nG$,
let $H_{a}$ be the largest convex subgroup of $G$ such that $a \notin H_{a} + nG$;
set $H_{a} = \{0\}$ if $a \in nG$.
Define
$\Auxc_n := G/\mathord{\sim}$, with $a \sim a'$ iff $H_{a} = H_{a'}$, and
let $\canc_n\colon G \surject \Auxc_n$ be the canonical map.
For $\alpha  = \canc_n(a) \in \Auxc_n$, define
$G_{\alpha} := H_{a}$.
\item
For $n \in \NN$ and $b \in G$, set
$H'_{b} := \bigcup_{\alpha \in \Auxc_n, b\notin G_\alpha} G_\alpha$,
where the union over the empty set is $\{0\}$.
Define
$\Auxe_n := G/\mathord{\sim}$, with $b \sim b'$ iff $H'_{b} = H'_{b'}$, and
let $\cane_n\colon G \surject \Auxe_n$ be the canonical map.
For $\alpha  = \cane_n(b) \in \Auxe_n$, define
$G_{\alpha} := H'_{b}$.
\item
For $n \in \NN$ and $\beta \in \Auxe_n$, define
$G_{\beta+} := \bigcap_{\alpha \in \Auxc_n, G_\alpha \supsetneq G_\beta} G_\alpha$,
where the intersection over the empty set is $G$.
Here, we view the index $\beta+$ as being an element of a copy of $\Auxe_n$
which we denote by $\Auxe^+_n$.
%For $n \in \NN$, define $\Auxe^+_n$ to be a copy of $\Auxe_n$
%whose elements are denoted by $\beta+$ for $\beta \in \Auxe_n$.
%For $\beta \in \Auxe_n$, set
%$G_{\beta+} := \bigcap_{\alpha \in \Auxc_n, G_\alpha \supsetneq G_\beta} G_\alpha$,
%where the intersection over the empty set is $G$.
\item
Define a total preorder on
$\bigdcup_{n \in \NN}
(\Auxc_n \dcup \Auxe_n \dcup \Auxe_n^+)$
by $\alpha \le \alpha'$ iff
$G_\alpha \subseteq G_{\alpha'}$.
Write $\alpha \asymp \alpha'$ if $G_\alpha = G_{\alpha'}$.
\end{enumerate}
\end{defn}

Definability (in $\Loag$) of the groups $G_\alpha$, $\alpha \in \Auxc_n$ is proven in
Lemma~\ref{lem:def-Hng}; once this is done,
it is clear that the new sorts are imaginary sorts of $\Loag$ and that all of the above is definable.

\begin{rem}\label{rem:def-e+}
If $b \ne 0$, then we have $G_{\cane_n(b)+} =
\bigcap_{\alpha \in \Auxc_n, b\in G_\alpha} G_\alpha$;
in particular, $G_{\cane_n(b)+}$ is strictly bigger than
$G_{\cane_n(b)}$, since $b \notin G_{\cane_n(b)}$.
(However, we might have $G_{\cane_n(0)+} = \{0\}$.)
\end{rem}

Fix $\alpha$ in any of the auxiliary sorts.
Recall that for each
quantifier free $\Lpres$-definable relation on $G/G_\alpha$,
we want the corresponding relation on $G$ to be quantifier free definable in
$\Lint$. If $G/G_\alpha$ is dense, then it suffices to put preimages
of the relations $=$, $<$, $\equiv_m$ into $\Lint$ (interpreted as ternary relations,
where $\alpha$ is the third operand). However, if
$G/G_\alpha$ has a minimal positive element, then
we need $\Lint$-predicates for preimages of $\Lpres$-relations
defined using this element.
We introduce the following notation for these predicates.

\begin{defn}\label{defn:quot-literals}
Suppose that $\alpha \in \Auxc_n \cup \Auxe_n \cup \Auxe^+_n$
for some $n \in \NN$ and that
$\pi\colon G \surject G/G_\alpha$ is the canonical projection.
For $\diamond \in \{=, <, >, \le, \ge, \equiv_m\}$,
write $x \diamrel_\alpha y$ if $\pi(x) \diamrel \pi(y)$ holds
in $G/G_\alpha$.

For $k \in \ZZ$, write $k_\alpha$ for $k$ times
the minimal positive element of $G/G_\alpha$ if $G/G_\alpha$
is discrete and set $k_\alpha := 0 \in G/G_\alpha$ otherwise.
Write $x \diamrel_\alpha y + k_\alpha$ for
$\pi(x) \diamrel \pi(y) + k_\alpha$.
\end{defn}

Note that $x \equiv_{m,\alpha} y$ holds iff $x - y \in G_\alpha + mG$.
%; from this point of view, they are just predicates
%for some more families definable subgroups of $G$.
We will need one additional kind of
predicates which is similar, but
where $G_\alpha$ is replaced by a group which looks rather technical.
For definability of that group and for more explanations, see Section~\ref{subsect:cong*}.

\begin{defn}\label{defn:cong*}
For $n, m, m' \in \NN$ and
$\alpha \in \Auxc_n \cup \Auxe_n \cup \Auxe^+_n$, set
\[
G^{[m']}_{\alpha} :=
\!\!\!\!\!\!\!\!\!\!\!\!
\bigcap_{\substack{H \subseteq G \text{ convex subgroup}\\H \supsetneq G_{\alpha}}}
\!\!\!\!\!\!\!\!\!\!\!\!
(H + m'G);
\]
write $x \equiv^{[m']}_{m,\alpha} y$ iff $x - y \in G^{[m']}_{\alpha} + mG$.
\end{defn}

A separate notation for $x - y \in G^{[m']}_{\alpha}$ is not needed, since
$G^{[m']}_{\alpha} = G^{[m']}_{\alpha} + m'G$.

Finally, in $\Lint$ we will need a few unary predicates on the auxiliary sorts:
one saying whether the group $G/G_\alpha$ is discrete, and
some predicates specifying the cardinalities of
certain quotients of two groups of the form $G_\alpha + pG$ or
$G^{[p^s]}_{\alpha} + pG$. Since $pG$ is contained in the denominator
of those quotients, they are $\FF_p$ vector spaces,
and specifying the cardinality is equivalent to specifying the dimension
over $\FF_p$.

Here is the complete definition of $\Lint$:

\begin{defn}
The language $\Lint$ consists of the following:
\begin{itemize}
\item
The main sort $G$ with the constant $0$, the binary function $+$, and
the unary function $-$.
\item
For each $p \in \PP$, the auxiliary sorts $\Auxc_p$, $\Auxe_p$ and $\Auxe^+_p$ from
Definition~\ref{defn:aux}.
\item
For each $p, p' \in \PP$: binary relations ``$\alpha \le \alpha'$\,''
on $(\Auxc_p \dcup \Auxe_p \dcup \Auxe^+_p) \times (\Auxc_{p'} \dcup \Auxe_{p'} \dcup \Auxe^+_{p'})$,
defined by $G_\alpha \subseteq G_{\alpha'}$. (For each $p,p'$, these
are nine relations.)
\item Predicates for the relations $x_1 \diamrel_\alpha x_2 + k_\alpha$ from Definition~\ref{defn:quot-literals},
where $\diamond \in \{=, <, \equiv_m\}$,
$k \in \ZZ$, $m \in \NN$, and where $\alpha$ may be from any of the
sorts $\Auxc_p$, $\Auxe_p$ and $\Auxe^+_p$.
(These are ternary relations
on $G \times G \times \Auxc_p$, $G \times G \times \Auxe_p$, and
$G \times G \times \Auxe^+_p$.)
\item
For each $p \in \PP$ and each $m, m' \in \NN$,
the ternary relation $x \equiv^{[m']}_{m,\alpha} y$
on $G \times G \times \Auxc_p$.
\item
For each $p \in \PP$, a predicate $\discr(\alpha)$ on $\Auxc_p$
which holds iff $G/G_\alpha$ is discrete.
\item
For each $p \in \PP$, each $s \in \NN$, and each $\ell \in \NN_0$, two predicates on
$\Auxc_p$ defining the sets
\begin{align*}
\{\alpha \in \Auxc_p &\mid \dim_{\FF_p} (G^{[p^s]}_{\alpha} + pG)/(G^{[p^{s+1}]}_{\alpha} + pG) = \ell\}
\rlap{$\quad$and}
\\
\{\alpha \in \Auxc_p &\mid\dim_{\FF_p} (G^{[p^s]}_{\alpha} + pG)/(G_{\alpha} + pG) = \ell\}
.
\end{align*}
\end{itemize}
\end{defn}

\begin{notation}
We write $\Main := \{G\}$ for the main sort and
$\Aux := \{\Auxc_p, \Auxe_p, \Auxe^+_p\mid p \in \PP\}$ for the collection
of auxiliary sorts. By abuse of notation, we will also
write $\Aux$ for the union of the auxiliary sorts.
We will write that a formula is ``$\Main$-qf'' if
it does not contain any quantifier running over a main sort variable.
\end{notation}

The usual predicates $<$ and
$\equiv_m$ on $G$ are $\Main$-qf $\Lint$-definable: they are equivalent to
$<_{\alpha_0}$ and $\equiv_{m, \alpha_0}$,
where $\alpha_0$ is the minimal element of, say, $\Auxc_2$.
The canonical map $\Auxe_p \to \Auxe^+_p, \alpha \mapsto \alpha+$
is easily $\Main$-qf definable from the preorder on  $\Auxe_p \dcup \Auxe^+_p$
using Remark~\ref{rem:def-e+}.
We will later see $\Main$-qf definability of the canonical maps
$\canc_p$, $\cane_p$ (Lemma~\ref{lem:can-eq}) and of
the analogues on $\Auxe_p$ and $\Auxe^+_p$ of the discreteness and dimension
predicates (Lemmas \ref{lem:bullet-def} and~\ref{lem:dim-def}).
Moreover, Lemmas \ref{lem:prime-c} and~\ref{lem:prime-e} will show
how to get along without having $\Auxc_n$, $\Auxe_n$, $\Auxe^+_n$, $\canc_n$,
and $\cane_n$ for arbitrary $n$.

Note that although $\Auxe_p$ and $\Auxe^+_p$ are in definable bijection,
identifying them would make the language pretty messy,
in particular because the
preorder on $\bigdcup_p(\Auxc_p \dcup \Auxe_p)$ is \emph{not} enough
to define the preorder on the whole of $\Aux$ in an $\Main$-qf way.

As announced, our main result is ``quantifier elimination relative to
the auxiliary sorts'', which is more than just elimination of main sort quantifiers.
Now let us make this precise; we first need a definition.

\begin{defn}\label{defn:fam-union}
Suppose that $L$ is any language, $T$ is an $L$-theory, $\Main \dcup \Aux$ is a partition of the sorts of $L$,
and $\phi(\xbar, \etabar)$ is an $L$-formula, where $\xbar$ are $\Main$-variables
and $\etabar$ are $\Aux$-variables.
We say that $\phi(\xbar, \etabar)$ is in \emph{family union form} if it is of the form
\[
\phi(\xbar, \etabar) =
\bigvee_{i = 1}^k
\exists \thetabar\, \big(\xi_i(\etabar, \thetabar) \wedge \psi_i(\xbar, \thetabar)\big)
\]
where $\thetabar$ are $\Aux$-variables,
the formulas $\xi_i(\etabar, \thetabar)$ live purely in the sorts $\Aux$,
each $\psi_i(\xbar, \thetabar)$ is a conjunction of literals (i.e., of
atoms and negated atoms), and $T$ implies that
the $L(\Aux)$-formulas
$\{\xi_i(\etabar, \alphabar) \wedge \psi_i(\xbar, \alphabar) \mid
1 \le i \le k, \alphabar \in \Aux\}$
are pairwise inconsistent.
\end{defn}

\begin{theorem}\label{thm:qe-int}
In the theory of ordered abelian groups,
each $\Lint$-formula
is equivalent to an $\Lint$-formula in family union form.
\end{theorem}

\begin{rem}\label{rem:polyhedron}
In $\Lint$, the formulas $\psi_i(\xbar, \thetabar)$ appearing in the family union form
are very simple. Without loss, each atom involves the main sort,
i.e., it is of the form $t(\xbar) \diamond_{\theta_\nu} t'(\xbar) + k_{\theta_\nu}$
where $t(\xbar), t'(\xbar)$ are $\ZZ$-linear combinations,
$\diamond \in \{=, <, \equiv_m, \equiv^{[m']}_{m}\}$, $\theta_\nu$ is one of the entries of $\thetabar$,
$k \in \ZZ$, and $m, m' \in \NN$ (where $k = 0$ if $\diamond$ is $\equiv^{[m']}_{m}$).
Moreover, ``$=$''-literals can be expressed using ``$<$'' and ``$>$'' instead.
Now the inequalities of $\psi_i$ define a convex polyhedron, and the remaining
literals ($\equiv_m, \not\equiv_m, \equiv^{[m']}_{m}, \not\equiv^{[m']}_{m}$)
are ``congruence conditions'' in the sense that each of them defines a set which
%is invariant under addition of elements of
consists of entire cosets of
$mG$ (possibly for several different $m \in \NN$).
%In fact, the set defined by the conjunction of all congruence conditions
%is $m_0G$-saturated if we take for $m_0$ the least common multiple of all $m$.
From this point of view, such sets are very similar to sets definable in $\Lpres$ by a conjunction of literals (which are also intersections of polyhedra with congruence conditions).
\end{rem}

\subsection{Definable functions are piecewise linear}

Using the above quantifier elimination theorem, it is easy to prove that
definable functions from $G^n$ to $G$ are piecewise linear. More precisely:

\begin{cor}\label{cor:fctn}
For any function $f\colon G^n \to G$ which is $\Loag$-definable with parameters
from a set $B$,
there exists a partition of $G^n$ into finitely many $B$-definable sets such that
on each such set $A$, $f$ is linear: there exist $r_{1}, \dots r_n,s \in \ZZ$ 
with $s\not=0$
and $b \in \dcl(B)$ such that for any $\abar \in A$, we have
$f(a_1, \dots, a_n) = \frac1s(\sum_i r_i a_i + b)$.
\end{cor}

Let us prove this right away, since it illustrates nicely how Theorem~\ref{thm:qe-int} can be applied.

\begin{proof}
Let $\phi(\xbar, y)$ be an $\Lint(B)$-formula in family union form defining the graph of $f$, let $\abar \in G^n$ be a tuple, set $c := f(\abar)$,
and consider $\phi(\abar, y) \in \Lint(B\cup\abar)$, which defines the single element set $\{c\}$. (We do not write the parameters from $B$ explicitly.)
Using a case distinction, we may suppose that the family union form of
$\phi(\abar, y)$ consists of a single family:
\[
\phi(\abar, y) = \exists \thetabar\, \big(\xi(\thetabar) \wedge \psi(\abar, y, \thetabar)\big)
.
\]
Let $\betabar$ be the (unique) tuple of $\Aux$ such that $G \models \psi(\abar, c, \betabar)$.

As in Remark~\ref{rem:polyhedron}, we may assume that $\psi(\abar, y, \betabar)$ uses no ``$=$''. Moreover, we may choose an $m_0 \in \NN$ such that
all congruence conditions of $\psi(\abar, y, \betabar)$ together define a union
of cosets of $m_0G$.

Using further case distinctions (which are definable in $\abar$), we can assume:
all literals of $\psi(\abar, y, \betabar)$ involve $y$ and
among these literals, there is at most one lower and one upper bound on $y$.

There has to be a lower bound; otherwise,
for $d \in G$ with $d > 0$, the element $c - m_0d$ would also satisfy
$\psi(\abar, y, \betabar)$.
We may suppose that the lower bound is of the form
$ry \vartriangleright_{\alpha} t(\abar) + k_{\alpha}$, where
$\mathord{\vartriangleright} \in \{>, \ge\}$, $\alpha \in \Aux$,
and where $t$ is a main sort term, i.e., a $\ZZ$-linear combination of entries of
$\abar$ plus an element of $\dcl(B)$.
If $G_\alpha \supsetneq \{0\}$, then again $c - m_0d$ satisfies
$\psi(\abar, y, \betabar)$ if we take $d \in G_\alpha, d > 0$; hence
$G_\alpha = \{0\}$. In particular, $k_\alpha$ can be seen as an element of $G$
(and not just as a notation). From this point of view, we have
$k_\alpha \in \dcl(\emptyset)$, so without loss, the lower bound is of the form
$ry \vartriangleright_{\alpha} t(\abar)$.

Since $c$ is unique satisfying $\psi(\abar, y, \betabar)$, it
must be the minimal element satisfying $ry \vartriangleright_{\alpha} t(\abar)$
and the congruence conditions. Such a minimum can only exist if $G_\alpha = \{0\}$.
If $G$ is dense, then $m_0G$ is dense in $G$, so a minimum has to be equal to the lower bound;
thus $\psi$ is equivalent to $ry = t(\abar)$ and we are done.
If $G$ is discrete, then we do a case distinction on the difference
$d := rc - t(\abar)$. This difference can be at most $rm_0 + 1$ (otherwise $c - m_0$ would
also satisfy $\psi(\abar, y, \betabar)$), so there are only finitely many cases.
Fixing $d$ is a definable condition on $\abar$ and for fixed $d$, $\psi$ is equivalent to
$rc = t(\abar) + d$, which again is linear.
\end{proof}

\subsection{A language with good syntactic properties}

For usual quantifier elimination, it suffices to prove that the quantifier
of $\exists x\,\psi(x)$ can be eliminated when $\psi(x)$ is quantifier free.
This does not work for relative quantifier elimination: neither if we only
try to get rid of $\Main$-quantifiers (then $\psi$ can contain $\Aux$-quantifiers,
which can make it pretty complicated), nor if we want to get a formula
in family union form (in that case, the main difficulty turns
out to be that it is not
clear whether formulas in family union form are closed under negation).
The following general result allows us to do such a reasoning anyway
under some syntactic assumptions about the language.
%To avoid this, we prove an (easy) general result which states that
%considering quantifier free $\phi(x)$ is enough under the assumption
%that in the language, the only connection between
%the main sort and the auxiliary sorts are functions from main to auxiliary.
%Moreover, under this assumption, elimination of main sort quantifiers
%automatically implies relative quantifier elimination in the sense of Theorem~\ref{thm:qe-int}.
%Here is the precise formulation.

\begin{prop}\label{prop:no-groups}
Let $L$ be a language and let $\Main \dcup \Aux$ be a partition of the sorts of $L$.
Suppose that the only symbols in $L$ connecting $\Main$ and $\Aux$
are functions from (products of) $\Main$-sorts to $\Aux$-sorts.
Let $T$ be an $L$-theory.

Consider a formula of the form
$\exists x\,\psi(x, \ybar, \etabar)$
where $x, \ybar$ are $\Main$-variables, $\etabar$ are $\Aux$-variables
and $\psi$ is quantifier free. Suppose that modulo $T$, any such
formula is equivalent to a formula without $\Main$-quantifiers.

Then modulo $T$, any $L$-formula is equivalent to
an $L$-formula in family union form.
\end{prop}

Note that the proposition does not require us to bring $\exists x\,\psi(x, \ybar, \etabar)$
into family union form; no $\Main$-quantifiers is enough.

To be able to apply this result to ordered abelian groups,
we introduce a second language $\Lsyn$ which has the required property:
all $\Lint$-predicates connecting $\Main$ and $\Aux$ will be replaced
by some predicates on $\Main$ and some functions from $\Main$ to $\Aux$.
Let me start by explaining the idea of how this can be done; a complete proof that
$\Lsyn$ is as strong as $\Lint$ will be given in Section~\ref{subsect:L2-vs-L3}.

%The disadavantage of $\Lsyn$ is that it is less intuitive than $\Lint$.
%In Section~\ref{subsect:L2-vs-L3}, we will prove that $\Lsyn$ is strong enough
%to define all $\Lint$-predicates without $\Main$-quantifiers; the idea is as follows.

The $\Lint$-predicates we have to get rid of are
$x_1 \diamrel_\eta x_2 + k_\eta$ for the various $\diamond$. First consider $x_1 =_\eta x_2$.
Since for fixed $x_1$ and $x_2$, $x_1 =_\eta x_2$ holds if and only if $\eta$ is
bigger than a certain bound depending only on $x_1 - x_2$,
we can replace the predicate $x_1 =_\eta x_2$ by the function from $G$ to $\Aux$ which
returns this bound. In the case $\eta \in \Auxc_p$, we already defined exactly this function:
it is the canonical map $\cane_p\colon G \surject \Auxe_p$;
for $\eta \in \Auxe_p \dcup \Auxe^+_p$, one verifies that $\cane_p$ still works.

A similar idea allows to express the predicates $x_1 \equiv_{p^r,\eta} x_2$ using
the canonical maps $\canc_{p^r}$ (for $p \in \PP$ and $r \in \NN$).
In principle, these maps go to $\Auxc_{p^r}$ which are not sorts of $\Lsyn$,
but we will see in Lemma~\ref{lem:prime-c} that $\Auxc_{p^r}$ and $\Auxc_{p}$ can be identified.

What is missing now is a way to deal with the predicates $x_1 \diamrel_\alpha x_2 + k_\alpha$ when $k \ne 0$
(for $\diamond \in \{=,\equiv_m\}$) and with $x \equiv^{[m']}_{m,\alpha} y$.
(The inequalities $<_\alpha$ are no problem.)
These predicates will essentially be replaced by their union over all $\alpha$.
We will see in Section~\ref{subsect:def1} how the $\Lint$-predicates can be reconstructed from this.

Here is the complete definition of the language $\Lsyn$:

\begin{defn}\label{defn:L2}
The language $\Lsyn$ consists of the following:
\begin{itemize}
\item
The main sort $G$ with $0$, $+$, $-$, $<$, and $\equiv_m$ (for $m \in \NN$).
\item
As in $\Lint$, the auxiliary sorts $\Auxc_p$, $\Auxe_p$ and $\Auxe^+_p$ with the
binary relations ``$\alpha \le \alpha'$\,''
on $(\Auxc_p \dcup \Auxe_p \dcup \Auxe^+_p) \times (\Auxc_{p'} \dcup \Auxe_{p'} \dcup \Auxe^+_{p'})$,
and on $\Auxc_p$ the unary predicates $\discr(\alpha)$,
$\dim_{\FF_p} (G^{[p^s]}_{\alpha} + pG)/(G^{[p^{s+1}]}_{\alpha} + pG) = \ell$, and
$\dim_{\FF_p} (G^{[p^s]}_{\alpha} + pG)/(G_{\alpha} + pG) = \ell$.
\item
For each $p \in \PP$ (and each $r \in \NN$): the canonical maps
$\canc_{p^r}\colon G \surject \Auxc_p$
and $\cane_p\colon G \surject \Auxe_p$ from Definition~\ref{defn:aux},
where $\Auxc_{p^r}$ is identified with $\Auxc_p$ using Lemma~\ref{lem:prime-c}.
\item
For each $k \in \ZZ \setminus \{0\}$: a unary predicate
``$x =_\bullet k_\bullet$'' on $G$ defined by: there exists a convex subgroup $H \subseteq G$
such that $G/H$ is discrete and the image of $x$ in $G/H$ is $k$ times the smallest
positive element of $G/H$; see Section~\ref{subsect:def1} for details,
in particular for definability.
\item
For each $m \in \NN$ and each $k \in \{1 , \dots, m - 1\}$:
a unary predicate ``$x \equiv_{m,\bullet} k_\bullet$'' on $G$ defined by:
there exists a convex subgroup $H \subseteq G$ such that $G/H$ is discrete and
the image of $x$ in $G/H$ is congruent modulo $m$ to $k$ times the minimal
positive element of $G/H$; again see Section~\ref{subsect:def1} for details.
\item
For each $p \in \PP$ and each $r,s \in \NN$ with $s \ge r$: a unary predicate
$D^{[p^s]}_{p^r}(x)$
on $G$ for:
there exists an $\alpha \in \Auxc_p$ such that
$x$ lies in $G^{[p^s]}_{\alpha} + p^rG$, but not in $G_\alpha + p^rG$.
\end{itemize}
\end{defn}

In this language, relative quantifier elimination will simply be the conclusion
of Proposition~\ref{prop:no-groups}:

\begin{theorem}\label{thm:qe-form}
In the theory of ordered abelian groups,
each $\Lsyn$-formula is equivalent to an $\Lsyn$-formula in family union form.
\end{theorem}

We will deduce Theorem~\ref{thm:qe-int} from this one by translating the $\Main$-qf $\Lsyn$-formula back into $\Lint$. This will be done at the end of
Section~\ref{subsect:L2-vs-L3}.

\subsection{Comparison to Gurevich and Schmitt}

Theorem~\ref{thm:qe-form} is very similar to the quantifier elimination results of Gurevich and Schmitt;
here we give a little translation table between our language $\Lsyn$ and the one
used in Schmitt's habilitation thesis \cite{Sch.habil}. The quantifier
elimination result of \cite{Sch.habil} (Lemma~4.3, Theorem~4.4) is also described in the introduction
of \cite{Sch.oAGmodComp} (Theorem~1.7).

Schmitt does not distinguish between the sorts
$\Auxc_n$, $\Auxe_n$, and $\Auxe_n^+$; instead, for each $n \in \NN$ he works with a single
sort $\Sp_n(G) := (\Auxc_n \dcup \Auxe_n \dcup \Auxe_n^+)/\mathord{\asymp}$ (the ``$n$-spine of $G$''),
with predicates for $\Auxc_n$ and $\Auxe_n$. (More precisely, Schmitt does not really use a
multi-sorted structure, but this is what his formulation boils down to.)

When eliminating the $\Main$-quantifiers of a given formula $\phi$, instead of
using several sorts $\Sp_p(G)$ for primes $p$,
he uses only one single sort $\Sp_n(G)$ for $n \in \NN$.

Instead of our dimension predicates, Schmitt has
predicates for the \emph{Szmielew-invariants} of
$G^{[p^r]}_{p^r,\alpha}/G_{p^r,\alpha}$ (see Definition on page~5 of \cite{Sch.habil}).
At first sight, it seems that the number of Szmielew-invariants is bigger than
the number dimensions for which we introduced predicates (for each $\alpha$, the
set of Szmielew-invariants is parametrized by two natural numbers, whereas we consider only
two families of dimensions parametrized by a single natural number), but a little
computation shows that many of the Szmielew-invariants are always equal
(and equal to our dimensions).

Finally, on the main sort, Schmitt has slightly different predicates than
our $x =_\bullet k_\bullet$, $x \equiv_{m,\bullet} k$ and $D^{[p^s]}_{p^r}$.

\section{Details of the languages}
\label{sect:details}

\subsection{The families of convex definable subgroups $G_\alpha$}

In Definition~\ref{defn:aux}, we introduced the families of convex groups
$G_\alpha$, but we still had to verify that they are definable in the case
$\alpha \in \Auxc_n$.

\begin{lem}\label{lem:def-Hng}
Fix $n \in \NN$.
For $a \in G$, the group $G_{\canc_n(a)}$ is definable uniformly in $a$.
\end{lem}
\begin{proof}
We may suppose $a \notin nG$. In that case,
$G_{\canc_n(a)}$ consists of those elements $b \in G$ such that
$a \notin \langle b \rangle\conv + nG$. The group $\langle b \rangle\conv$ is
not definable in general, but we have
$\langle b \rangle\conv + nG = [0, n |b|] + nG$, which is definable;
here, $|b|$ denotes the absolute value of $b$.
\end{proof}

We defined the sorts $\Auxc_n$, $\Auxe_n$ and $\Auxe^+_n$ for arbitrary $n$,
but in our languages, we only have them for $n$ prime.
The following two lemmas will allow us to reduce any usage of these sorts
to the prime cases. In particular, we show that $\Auxc_{p^r}$ can be identified
with $\Auxc_{p}$, as required in the definition of $\Lsyn$.

We use the notation ``$p^r\mid\mid n$'' from number theory which means
that $p$ is a prime divisor of $n$ and that $p^r$ is the maximal power of $p$ dividing $n$.

\begin{lem}\label{lem:prime-c}
Let $n \in \NN$.
\begin{enumerate}
\item
%There is a bijection $\Auxc_{p^r} \to \Auxc_{p}$.
We have the following equality of sets of
convex subgroups of $G$:
\[
\{G_{\alpha} \mid \alpha \in \Auxc_n\}
\,=
\bigcup_{p \in \PP, p \mid n}\{G_{\alpha} \mid \alpha \in \Auxc_p\}
.
\]
In particular, there is a (unique, definable) bijection $\Auxc_{p^r} \to \Auxc_{p}$
which is compatible with $\alpha \mapsto G_\alpha$.
\item
For any $a \in G$, we have
\[
G_{\canc_n(a)} = \bigcup_{p^r\mid\mid n} G_{\canc_{p^r}(a)}
.
\]
In particular, $\canc_{n}(a) \asymp \max_{p^r\mid\mid n}\canc_{p^r}(a)$.
\end{enumerate}
\end{lem}

\begin{proof}
We start with (1)~``$\supseteq$''; more precisely,
for $m \mid n$, we prove
$\{G_\alpha \mid \alpha \in \Auxc_n\} \supseteq
\{G_\alpha \mid \alpha \in \Auxc_m\}$.
Consider $G_\alpha \ne \{0\}$ in the right hand set and choose $a \in G \setminus mG$ with $\alpha = \canc_m(a)$.
Recall that $G_\alpha$ is the largest convex subgroup of $G$
with $a \notin G_\alpha + mG$. For any convex subgroup $H \in G$,
we have $a \in H + mG$ if and only if $a' := \frac nm a \in H + nG$; hence
$G_\alpha = G_{\canc_n(a')}$.

Next, we prove (2).
The inclusion ``$\supseteq$'' is clear. For ``$\subseteq$'',
we may suppose that $a \in G \setminus nG$. By the Chinese
remainder theorem, we have
$G_{\canc_n(a)} + nG = \bigcap_{p^r \mid\mid n} (G_{\canc_n(a)} + p^rG)$,
so $a \notin G_{\canc_n(a)} + nG$ implies $a \notin G_{\canc_n(a)} + p^rG$
for some $p \mid n$. This in turn implies
$G_{\canc_n(a)} \subseteq G_{\canc_{p^r}(a)}$.

Finally, we prove (1)~``$\subseteq$''.
By (2), we have $\{G_\alpha \mid \alpha \in \Auxc_n\} \subseteq \{G_\alpha \mid \alpha \in \bigcup_{p^r \mid\mid n}\Auxc_{p^r}\}$, so it suffices to do the case where
$n = p^r$.
Suppose that $\alpha = \canc_{p^r}(a)$ for some $a \in G \setminus p^rG$ and consider the group $G_{\alpha}$ from the left hand set of (1).
Let $s \in \NN$ be maximal with
$a \in G_{\alpha} + p^sG$; by assumption $s < r$.
Write $a = b + p^sa'$ for $b \in G_{\alpha}$ and $a'\in G$.
Then $a' \notin G_{\alpha} + pG$, since otherwise
$b + p^sa' \in G_{\alpha} + p^{s+1}G$.
On the other hand for any convex subgroup $H$ strictly larger
than $G_{\alpha}$, we have $b + p^sa' = a \in H + p^rG \subseteq H + p^{s+1}G$,
so $p^sa' \in H + p^{s+1}G$, so $a' \in H + pG$. Hence
$G_{\canc_p(a')} = G_{\alpha}$.
\end{proof}

\begin{lem}\label{lem:prime-e}

For any $n \in \NN$ and any $a \in G$, we have
\[
G_{\cane_n(a)} = \bigcup_{p \in \PP, p \mid n} G_{\cane_{p}(a)}
\quad\text{and}\quad
G_{\cane_n(a)+} = \bigcap_{p \in \PP, p \mid n} G_{\cane_{p}(a)+}
.
\]
In particular, $\cane_{n}(a) \asymp \max_{p \in \PP, p\mid n}\cane_{p}(a)$
and $\cane_{n}(a)+ \asymp \min_{p \in \PP, p\mid n}(\cane_{p}(a)+)$.
\end{lem}
\begin{proof}
This follows directly from Lemma~\ref{lem:prime-c}~(1), where
for $G_{\cane_n(a)+}$ we use Remark~\ref{rem:def-e+}.
\end{proof}

\subsection{Congruence conditions and expressing $\canc_n$ and $\cane_n$ in $\Lint$}
\label{subsect:cong*}

%First, a general remark on definable subgroups and quantifier elimination.
%If $H_1$ and $H_2$ are definable without quantifiers, then so is $H_1 \cap H_2$,
%but defining $H_1 + H_2$ might need a quantifier. That explains why

In Definition~\ref{defn:cong*}, we introduced the group
$G^{[n]}_{\alpha} = \bigcap_{H \supsetneq G_{\alpha}} (H + nG)$
for $n \in \NN$ and $\alpha \in \Aux$. The point is that $G^{[n]}_{\alpha}$
might be strictly bigger than $(\bigcap_{H \supsetneq G_{\alpha}} H) + nG$,
and in general, it is not of the form $H_0 + nG$ for any convex subgroup $H_0 \subseteq G$
(see example in Section~\ref{subsect:ex[n]}).
We will need these groups to
express the $\Lsyn$-function $\canc_n$ in $\Lint$ without $\Main$ quantifiers;
this will be done at the end of this section.

The following lemma gives an equivalent definition of $G^{[n]}_{\alpha}$
(using not all convex subgroups of $G$) which in particular shows that
it is definable.

\begin{lem}\label{lem:cong*}
Let $n \in \NN$.
\begin{enumerate}
\item
For any convex subgroup $H \subseteq G$, we have
\[
H + nG = \bigcap_{\substack{\alpha' \in \Auxc_n\\G_{\alpha'} \supseteq H}} (G_{\alpha'} + nG)
.
\]
\item
For $\alpha \in \Aux$, we have
\[
G^{[n]}_{\alpha} = \bigcap_{\substack{\alpha' \in \Auxc_n\\\alpha' > \alpha}} (G_{\alpha'} + nG)
.
\]
\end{enumerate}
\end{lem}

\begin{proof}
(1)
``$\subseteq$'' is clear, so suppose now $a \notin H + nG$.
Set $\alpha' = \canc_n(a)$. Then by definition $a \notin G_{\alpha'} + nG$.

(2)
Again, ``$\subseteq$'' is clear.
By applying (1) to the groups $H + nG$ appearing in the definition of
$G^{[n]}_{\alpha}$, we obtain that $G^{[n]}_{\alpha}$ is the intersection
of groups $G_{\alpha'} + nG$ for some $\alpha' \in \Auxc_n$.
Since $G_{\alpha} \subsetneq H \subseteq G_{\alpha'}$, these $\alpha'$ satisfy
$\alpha' > \alpha$.
\end{proof}

The relations $\equiv_{m,\alpha}$ and $\equiv^{[n]}_{m,\alpha}$ have a lot of
similar basic properties. The following three lemmas list those
which we will need; we formulate them in terms of the groups
$G_\alpha + mG$ and $G^{[n]}_\alpha + mG$.

\begin{lem}\label{lem:m-mid-n}
For $\alpha \in \Aux$ and $m,n\in \NN$, we have
\[
G^{[n]}_\alpha + mG = G^{[n]}_\alpha + \gcd(m,n)G.
\]
In particular, in $\Lint$ we only need those predicates
$\equiv^{[n]}_{m,\alpha}$ with $m \mid n$.
\end{lem}
\begin{proof}
Since $nG \subseteq G^{[n]}_\alpha$, the left hand side contains
$nG + mG = \gcd(m,n)G$.
\end{proof}

\begin{lem}\label{lem:cong-mult}
For $k\in\ZZ$, $m, n \in \NN$, and $\alpha \in \Aux$, we have:
\[
\begin{aligned}
k(G_{\alpha} + mG) &= kG \cap (G_{\alpha} + kmG)\\
k(G^{[n]}_{\alpha} + mG) &= kG \cap (G^{[kn]}_{\alpha} + kmG)
\end{aligned}
\]
\end{lem}
\begin{proof}
Straight forward, using
that the convexity of $G_\alpha$ implies $kG_\alpha = kG \cap G_\alpha$
and using the definition of $G^{[n]}_\alpha$.
\end{proof}

\begin{lem}\label{lem:cong-factorize}
Suppose that $m=m_1\cdot m_2, n = n_1\cdot n_2 \in \NN$ with $m_1,m_2$ coprime and
$n_1,n_2$ coprime, and suppose that $\alpha \in \Aux$. Then we have:
\[
\begin{aligned}
G_{\alpha} + mG &= (G_{\alpha} + m_1G) \cap (G_{\alpha} + m_2G)\\
G^{[n]}_{\alpha} + mG &= (G^{[n]}_{\alpha} + m_1G) \cap (G^{[n]}_{\alpha} + m_2G)\\
G^{[n]}_{\alpha} + mG &= (G^{[n_1]}_{\alpha} + mG) \cap (G^{[n_2]}_{\alpha} + mG)
\end{aligned}
\]
\end{lem}
\begin{proof}
The first two equations 
are simply the Chinese remainder theorem in the groups $G/G_\alpha$
and $G/G^{[n]}_\alpha$, respectively.
The third one also follows directly from the Chinese remainder theorem,
but since this is slightly more subtle, let us write down the details.
``$\subseteq$'' is clear. For ``$\supseteq$'',
suppose that $a$ is an element of the right hand side,
i.e., there are elements $b_i \in mG$,  $c_{\alpha',i} \in G_{\alpha'}$
and $d_{\alpha',i} \in G$ such that for $i=1,2$ and for all $\alpha' > \alpha$ we have
\[
a = b_i + c_{\alpha',i} + n_id_{\alpha',i}
.
\]
Find $x_1,x_2 \in \ZZ$ with $x_1n_1 + x_2n_2 = 1$.
Then
\[
\begin{split}
a &=
x_1n_1(b_2 + c_{\alpha',2} + n_2d_{\alpha',2}) +
x_2n_2(b_1 + c_{\alpha',1} + n_1d_{\alpha',1})\\
&= \underbrace{x_1n_1b_2 + x_2n_2b_1}_{\in\,mG} +
\underbrace{x_1n_1c_{\alpha',2} + x_2n_2c_{\alpha',1}}_{\in\,G_{\alpha'}} + \underbrace{n_1n_2(x_1d_{\alpha',2} + x_2d_{\alpha',1})}_{\in\,nG}
,
\end{split}
\]
i.e., $a \in G^{[n]}_{\alpha} + mG$.
\end{proof}

Let us end this section by relating the $\Lsyn$-maps $\canc_n$ and $\cane_n$ with
the $\Lint$-predicates $\equiv_{n,\alpha}$ and $=_{\alpha}$.

\begin{lem}\label{lem:can-eq}
For $n \in \NN$, $a \in G$, $\alpha \in \Aux$ and $\beta \in \Auxc_n \cup \Auxe_n$, we have the following equivalences,
where for $\overset{\textup{(1)}}{\Longrightarrow}$, we additionally need
$a \notin nG$, and for
$\overset{\textup{(3)}}{\Longrightarrow}$, we additionally need
$a \ne 0$.
\begin{align*}
\canc_n(a) \ge \alpha &\overset{\textup{(1)}}{\ciff} a \not\equiv_{n,\alpha} 0
&
\cane_n(a) \ge \beta &\overset{\textup{(3)}}{\ciff} a \ne_{\beta} 0
\\
\canc_n(a) \le \alpha &\overset{\textup{(2)}}{\iff} a \equiv^{[n]}_{n,\alpha} 0
&
\cane_n(a) \le \beta &\overset{\textup{(4)}}{\iff} a =_{\beta+} 0
.
\end{align*}
\end{lem}
\begin{proof}
(1)
For any convex subgroup $H \subseteq G$, we have
the equivalence
$G_{\canc_n(a)} \supseteq H
\iff a \notin H + nG$, where for ``$\Longrightarrow$'', we additionally assume
$a \notin nG$. Set $H := G_\alpha$.

(2) If $a \in nG$, then both sides are true anyway. Otherwise, (2) follows from (1)
using that the right hand side is equivalent to
$a \equiv_{n,\alpha'} 0$ for all $\alpha' > \alpha, \alpha' \in \Auxc_n$
by Lemma~\ref{lem:cong*}~(2).

(3) If $H \subseteq G$ is a union of groups of the form $G_\alpha$ for
$\alpha \in \Auxc_n$, then we have the equivalence
$G_{\cane_n(a)} \supseteq H
\iff a \notin H$, where for ``$\Longrightarrow$'', we additionally assume
$a \ne 0$. Set $H := G_\beta$.

(4) Again, for $a = 0$ both sides are true anyway and for $a \ne 0$,
the statement follows from (3).
\end{proof}

\subsection{More dimensions of $\FF_p$-vector spaces}
\label{subsect:dim}

In the definition of $\Lint$, we added predicates for the dimension
as $\FF_p$-vector spaces of certain quotients of groups
of the form $G_\alpha + pG$ or $G^{[p^s]}_\alpha + pG$; in particular,
we required $\alpha \in \Auxc_p$. The following lemma shows that this
is enough to get the dimension of arbitrary quotients of two groups
of this type, and for any $\alpha \in \Aux$.
Moreover, we also want to consider the quotient of $G$ by such a group.
To simplify formulating the lemma, we temporarily introduce the following notation.

\begin{notation}
Set $G^{[p^\infty]}_\alpha := G_\alpha$ and $G_\infty := G$.
\end{notation}

Note that all groups we are interested in form a long chain:
for $\alpha, \alpha' \in \Aux$ with $\alpha < \alpha'$, we have
\begin{align*}
&\dots\subseteq
G^{[p^\infty]}_\alpha + pG
\subseteq \dots
%\subseteq G^{[p^3]}_\alpha + pG
\subseteq G^{[p^2]}_\alpha + pG \subseteq
G^{[p]}_\alpha + pG
\subseteq \dots\\
&\dots \subseteq
G^{[p^\infty]}_{\alpha'} + pG
\subseteq \cdots \cdots \subseteq
G_\infty.
\end{align*}
% \[
% \begin{array}{ccccc}
% G^{[p^\infty]}_\alpha + pG
% &\subseteq \dots \subseteq& G^{[p^2]}_\alpha + pG &\subseteq&
% G^{[p]}_\alpha + pG
% \\
% ||&&&&||\\
% G_\alpha + pG &&&& G^{[p]}_\alpha
% \end{array}
% \]
Thus taking a quotient $(G^{[p^{s_2}]}_{\alpha_2} + pG)/(G^{[p^{s_1}]}_{\alpha_1} + pG)$
makes sense iff
\[\tag{*}
\alpha_1 < \alpha_2 \quad \vee \quad (\alpha_1 \asymp \alpha_2 \wedge s_1 \ge s_2)
\]
holds.

\begin{lem}\label{lem:dim-def}
Fix $p \in \PP$, $s_1, s_2 \in \NN \cup \{\infty\}$, and $\ell \in \NN_0$,
and fix two auxiliary sorts $\Aux_1$ and $\Aux_2$. We additionally allow
$\Aux_2 = \{\infty\}$. Then the set
\[
\{(\alpha_1, \alpha_2) \in \Aux_1 \times \Aux_2 \mid \text{\textup{(*)} holds and }\dim_{\FF_p}
(G^{[p^{s_2}]}_{\alpha_2} + pG)/(G^{[p^{s_1}]}_{\alpha_1} + pG) = \ell
\}
\]
is $\Main$-qf definable in $\Lint$.
\end{lem}

\begin{proof}
Set $H_i := G^{[p^{s_i}]}_{\alpha_i} + pG$ for $i = 1, 2$.
To obtain definability of the dimension of $H_2/H_1$ (in the above sense),
it suffices to find some intermediate groups such that
the dimensions of successive quotients are definable in the same sense;
we will use this method to reduce to dimensions
which are given by $\Lint$-predicates.

We will use Lemma~\ref{lem:cong*} several times to show that some groups of the form
$G_{\alpha} + pG$ or $G^{[p]}_{\alpha}$ are equal. By that lemma, such groups
are intersections of groups $G_\beta + pG$ for some $\beta \in \Auxc_p$ (note that we do
not require $\alpha \in \Auxc_p$), so we get equality
as soon as the corresponding sets of $\beta$ are equal.

%We have $G_{\alpha_i} + pG \subseteq H_i \subseteq G^{[p]}_{\alpha_i}$.
%Using Lemma~\ref{lem:cong*}, we can write $G_{\alpha_i} + pG$ and $G^{[p]}_{\alpha_i}$
%as intersections of groups $G_\beta$ for some $\beta \in \Auxc_p$;
%in particular, such groups are equal when the intersections go over the same sets of %$\beta$. We will use this several times.

Suppose first that $\alpha_1 \asymp \alpha_2$.
If there is no $\beta \in \Auxc_p$ with $\alpha_i \asymp \beta$, then
$G_{\alpha_i} + pG = G^{[p]}_{\alpha_i}$ (by Lemma~\ref{lem:cong*}),
which implies $H_1 = H_2$, since $G_{\alpha_i} + pG \subseteq H_i \subseteq G^{[p]}_{\alpha_i}$.
Thus we may suppose
that $\alpha_i \asymp \beta$ for some $\beta \in \Auxc_p$.
Moreover, we may suppose $s_1 > s_2$. If
$s_1 = \infty$, then ``$\dim_{\FF_p} H_2/H_1 = \ell$'' itself is a predicate of
$\Lint$; otherwise compute the dimension using the chain of groups
\[
H_1
\,\,\subseteq\,\, G^{[p^{s_1-1}]}_{\beta}+pG
\,\,\subseteq\,\, G^{[p^{s_1-2}]}_{\beta}+pG
\,\,\subseteq \dots \subseteq\,\,
H_2
.
\]

Now consider the case $\alpha_1 < \alpha_2$.
Set
$I := \{\beta \in \Auxc_p \mid \alpha_1 < \beta < \alpha_2\}$. We claim that if $I$ has cardinality bigger than $\ell+1$, then
$\dim_{\FF_p} H_2/H_1 > \ell$.
Indeed, for any $\beta,\beta' \in I$ with $\beta < \beta'$, if we take $a \in G$
with $\canc_p(a) = \beta$, we have $a \in (G_{\beta'} +pG) \setminus (G_{\beta} +pG)$
so we get a strictly ascending chain of more than $\ell + 1$ groups between $H_1$ and $H_2$.

Finally, if $I = \{\beta_1, \dots, \beta_k\}$,
we use the following chain to compute the dimension:
\[\begin{split}
H_1
\,\,\subseteq\,\,
G^{[p]}_{\alpha_1} \,=\, G_{\beta_1} + pG
\,\,\subseteq\,\,
G^{[p]}_{\beta_1} \,=\, G_{\beta_2} + pG
\,\,\subseteq\cdots\\
\dots\subseteq\,\,
G^{[p]}_{\beta_k} \,=\, G_{\alpha_2} + pG
\,\,\subseteq\,\,
H_2
\end{split}
\]
Here, all equalities follow from Lemma~\ref{lem:cong*}, the dimension
at the first and the last step have already been computed above, and the remaining
dimensions are given by $\Lint$-predicates.
\end{proof}

\subsection{The predicates $x =_\bullet k_\bullet$, $x \equiv_{m,\bullet} k_\bullet$
and $D^{[p^s]}_{p^r}(x)$}
\label{subsect:def1}

The $\Lsyn$-predicates
$x =_\bullet k_\bullet$ and $x \equiv_{m,\bullet} k_\bullet$
were defined using quantification over all convex subgroups $H$
of $G$ such that $G/H$ is discrete. The following lemma shows
that this is definable.

\begin{lem}\label{lem:bullet-def}
If $H \subseteq G$ is any convex subgroup such that $G/H$ is discrete,
then in each of the sorts $\Auxc_n, \Auxe_n$, $n \ge 2$,
there exists an $\alpha$ with
$H = G_\alpha$. In particular:
\begin{enumerate}
\item
$x =_\bullet k_\bullet$ and $x \equiv_{m,\bullet} k_\bullet$
are definable in $\Loag$.
\item
In any auxiliary sort, the set of $\alpha$ such that $G/G_\alpha$ is discrete
is $\Main$-qf definable (both, in $\Lint$ and in $\Lsyn$).
\end{enumerate}
\end{lem}
\begin{proof}
Suppose that $G/H$ is discrete, and choose any $a \in G$ in the preimage
of the smallest positive element of $G/H$. Then $a \notin H + nG$ for any
$n \ge 2$, but $a \in H' \subseteq H' + nG$ for any convex $H' \supsetneq H$; hence
$H = G_{\canc_n(a)}$. Moreover, since $a \in H' \setminus H$, we also have
$H = G_{\cane_n(a)}$.

(1) $x =_\bullet k_\bullet$ iff $\exists (\alpha \in \Auxc_2) \, (\discr(\alpha)
\wedge x =_\alpha k_\alpha)$; and similarly for $\equiv_{m,\bullet}$.

(2) $G/G_\alpha$ is discrete iff $\exists(\beta \in \Auxc_2) \, (\beta \asymp \alpha \wedge \discr(\beta))$.
\end{proof}

The following lemma shows the connection between
the $\Lsyn$-predicates $x =_\bullet k_\bullet$,
$x \equiv_{m,\bullet} k_\bullet$ and $D^{[p^s]}_{p^r}$
and the corresponding $\Lint$-predicates.
Each of these $\Lsyn$-predicates defines a union of some sets
$X_\alpha$ given by the corresponding $\Lint$-predicate,
where $\alpha$ runs through a certain auxiliary set $\Xi$.
The point is that if $x$ lies in this union, then $\alpha$ can be recovered from $x$
by a definable function form the union to $\Xi$.
This will allow us
to define the sets $X_\alpha$ using the corresponding $\Lsyn$-predicate.

\begin{lem}\label{lem:L2-unions}
For $x \in G$ we have the following implications (1a), (2a), (3a),
which in particular imply the equivalences (1b), (2b), (3b).\\
\textup{(1)} For $k \in \ZZ \setminus \{0\}$ and $\alpha \in \Aux$:
\begin{gather*}
\discr(\alpha) \wedge x =_\alpha k_\alpha
\Longrightarrow \alpha \asymp \cane_2(x)
\tag{1a}
\\
x =_\bullet k_\bullet
\iff
\discr(\cane_2(x)) \wedge x =_{\cane_2(x)} k_{\cane_2(x)}
\tag{1b}
\end{gather*}
\textup{(2)} For $m \in \NN, k \in \{1, \dots, m-1\}$ and $\alpha \in \Aux$:
\begin{gather*}
\discr(\alpha) \wedge x \equiv_{m,\alpha} k_{\alpha}
\Longrightarrow \alpha \asymp \canc_m(x)
\tag{2a}
\\
x \equiv_{m,\bullet} k_\bullet
\iff
\discr(\canc_m(x)) \wedge x \equiv_{m,\canc_m(x)} k_{\canc_m(x)}
\tag{2b}
\end{gather*}
\textup{(3)} For $p \in \PP$, $r,s\in \NN$ with $s \ge r$ and $\alpha \in \Auxc_{p^r}$:
\begin{gather*}
x \equiv^{[p^s]}_{p^r,\alpha} 0 \wedge x \not\equiv_{p^r,\alpha} 0
\Longrightarrow \alpha  = \canc_{p^r}(x)
\tag{3a}
\\
D^{[p^s]}_{p^r}(x)
\iff
x \equiv^{[p^s]}_{p^r,\canc_{p^r}(x)} 0 \wedge x \not\equiv_{p^r,\canc_{p^r}(x)} 0
\tag{3b}
\end{gather*}
\end{lem}
\begin{rem}
The map $\cane_2$ in (1) can of course be replaced by any other
map $\cane_p$, $p \in \PP$.
%The right hand side of (2) uses $\canc_m$, which is not in the languages, but
%by Lemma~\ref{lem:prime-c},
%it can be rewritten using $\canc_p$ for $p \mid m$.
\end{rem}
\begin{proof}[Proof of Lemma~\ref{lem:L2-unions}]
In (1a) and (2a), discreteness of $G/G_\alpha$ and
the choice of $k$ ensures that the left
hand side implies $x \notin G_\alpha$ (and even
$x \notin G_\alpha + mG$ in the case of (2a)). On the other hand, we have
$x \in H$ for any convex group $H \supsetneq G_\alpha$.
This implies the corresponding right hand side.
For (3a), use $x \equiv^{[p^s]}_{p^r,\alpha} 0 \Rightarrow x \equiv^{[p^r]}_{p^r,\alpha} 0$
and Lemma~\ref{lem:can-eq}.

In (Xb), $x$ satisfies the left hand side
if and only there is an $\alpha$ (in $\Auxe_2$, $\Auxc_m$ or $\Auxc_{p^r}$,
respectively)
such that $x$ satisfies
the left hand side of (Xa). The right hand side of (Xa) says how
this $\alpha$ can be obtained from $x$. Plugging this in yields the
right hand side of (Xb).
\end{proof}

\subsection{Translation between $\Lsyn$ and $\Lint$}
\label{subsect:L2-vs-L3}

When introducing the language $\Lsyn$, we claimed that it is strong enough
to express $\Lint$ without $\Main$-quantifiers.
On the other hand, we want to deduce quantifier elimination in $\Lint$ from
quantifier elimination in $\Lsyn$, hence we also
need (a version of) the other direction.
This is what we prove in this section.
At the end of the section, the translation $\Lsyn \rightsquigarrow \Lint$ will
be applied to deduce
Theorem~\ref{thm:qe-int} from Theorem~\ref{thm:qe-form}.

\begin{prop}\label{prop:L3-to-L2}
Any $\Lint$-predicate can be expressed in $\Lsyn$ without $\Main$-quantifiers.
\end{prop}

\begin{rem}
Since any function symbol in $\Lint$ is also contained in $\Lsyn$,
this implies that any $\Main$-qf $\Lint$-formula
is equivalent to an $\Main$-qf $\Lsyn$-formula.
\end{rem}

\begin{proof}[Proof of Proposition~\ref{prop:L3-to-L2}]
The predicates of $\Lint \setminus \Lsyn$ are the following:
\begin{itemize}
\item
$x \diamrel_\eta y + k_\eta$
where $\diamond \in \{=, <, \equiv_m\}$,
$k \in \ZZ$, $m \in \NN$, and where $\eta$ may be from any of the
sorts of $\Aux$;
\item
$x \equiv^{[n]}_{m,\eta} y$ for $m, n \in \NN$ and where $\eta$ is
from one of the sorts $\Auxc_p$.
\end{itemize}
Concerning
$x \diamrel_\eta y + k_\eta$, if $k \ne 0$, then we may assume that
$G/G_\eta$ is discrete, since otherwise by definition $k_\eta = 0$.
(Recall that this discreteness is definable on any auxiliary sort by
Lemma~\ref{lem:bullet-def}.)

We now translate all these predicates into $\Lsyn$, starting with the easier ones
so that we can use them for the more difficult ones.

First consider $x \equiv_{m,\eta} y$ (for $\eta$ from any $\Aux$-sort).
By Lemma~\ref{lem:can-eq}~(1), this
is equivalent to $\canc_m(x - y) < \eta \vee x \equiv_m y$, which
is equivalent to $\bigwedge_{p^r\mid\mid m}\canc_{p^r}(x - y) < \eta \vee x \equiv_m y$
by Lemma~\ref{lem:prime-c}.

Next consider $x =_\eta y$. If $\eta \in \Auxc_p \cup \Auxe_p$,
then by Lemma~\ref{lem:can-eq}~(3) this is equivalent to
$\cane_p(x - y) < \eta \vee x = y$.
If $\eta \in \Auxe^+_p$,
then it is equivalent to
$\forall (\theta \in \Auxc_p)\,(\theta \ge \eta \rightarrow x =_\theta y)$.

Now consider $x =_\eta y + k_{\eta}$ for $k \ne 0$.
(Recall that we assume now that $G/G_\eta$ is discrete.)
Then Lemma~\ref{lem:L2-unions} (1a) implies $\eta = \cane_2(x - y)$,
and under this assumption,
Lemma~\ref{lem:L2-unions} (1b) implies that $x =_\eta y + k_{\eta}$
is equivalent to $x - y =_\bullet k_\bullet$.
Thus (under the assumption $\discr(\eta)$):
\[
x =_\eta y + k_{\eta} \iff
\eta = \cane_2(x - y) \wedge
x - y =_\bullet k_\bullet
.
\]

Exactly the same argument yields, for $m \in \NN$ and $k \in \{1, \dots, m-1\}$
(which we may assume):
\[
x \equiv_{m,\eta} y + k_{\eta} \iff \eta = \canc_m(x - y) \wedge
x - y \equiv_{m,\bullet} k_\bullet
.
\]

Concerning $x \equiv^{[n]}_{m,\eta} y$, we may assume that $m$ and $n$
are prime powers by Lemma~\ref{lem:cong-factorize}, and we may assume $m \mid n$
by Lemma~\ref{lem:m-mid-n}; so $m = p^r$ and $n = p^s$ for some $p \in \PP$
and $s \ge r$. Moreover, it suffices to define
$x \equiv^{[p^s]}_{p^r,\eta} y \wedge \neg x \equiv_{p^r,\eta} y$;
this again works in the same way as before with Lemma~\ref{lem:L2-unions},
yielding:
\[
x \equiv^{[p^s]}_{p^r,\eta} y
\iff
x \equiv_{p^r,\eta} y
\vee
(\eta = \canc_{p^r}(x - y) \wedge
D^{[p^s]}_{p^r}(x - y)
).
\]

Finally, consider $x <_\eta y + k_\eta$.
If $k = 0$, then this is equivalent to $x < y \wedge x \ne_\eta y$.
If $k$ is positive, then we take the disjunction of this with
$x =_\eta y + i_\eta$ for $0 \le i < k$;
if $k$ is negative, then we take the conjunction of this
with $x \ne_\eta y + i_\eta$ for $k \le i < 0$.
\end{proof}

\begin{prop}\label{prop:L2-to-L3}
Every quantifier free $\Lsyn$-formula
is equivalent to an $\Lint$-formula
in family union form.
\end{prop}

\begin{proof}
Let $\phi(\xbar, \etabar)$ be a given quantifier free $\Lsyn$-formula;
we have to get rid of the following kind of atoms:
\begin{enumerate}
\item
$t_1 \diamrel t_2$ where $\diamond \in \{<,>, \equiv_m\}$ and $t_1$, $t_2$ are
main sort terms (and $m \in \NN$);
\item
$t =_\bullet k_\bullet$, $t \equiv_{m,\bullet} k_\bullet$
and $D^{[p^s]}_{p^r}(t)$, where $t$ is a main sort term (and $p \in \PP$, $m, r, s \in \NN$);
\item
atoms involving $\canc_{p^r}(t)$ or $\cane_p(t)$, where $t$ is a main sort term
(and $p \in \PP, r \in \NN$).
\end{enumerate}
An atom $t_1 \diamrel t_2$ of type (1) can be replaced by $t_1 \diamrel_{\canc_2(0)} t_2$.
To get rid of the atoms of type (2), apply Lemma~\ref{lem:L2-unions} (1b), (2b), (3b).
It remains to get rid of the functions $\canc_m$ and $\cane_p$ (for $m \in \NN$, $p \in \PP$) (including the newly introduced ones) and bring the formula into family union form.

Let $\tau_i(\xbar)$ be the terms of
$\phi$ which are of the form
$\canc_{m}(t(\xbar))$ or $\cane_p(t(\xbar))$, where $m \in \NN$, $p \in \PP$, and where $t(\xbar)$ is a main sort term. We replace $\phi$ by the equivalent formula
\[
\exists \thetabar\, \big((\bigwedge_i\tau_i(\xbar) = \theta_i) \wedge
    \phi[\textstyle{\frac{\theta_i}{\tau_i(\xbar)}}]_{i}\big).
\]
(Here, the notation $\phi[\frac rs]$ means: the formula obtained from $\phi$
by replacing all occurrences of $s$ by $r$.)
The atoms $\tau_i(\xbar) = \theta_i$ can be expressed in $\Lint$ using Lemma~\ref{lem:can-eq}:
$\canc_{m}(t(\xbar)) = \theta$ is equivalent to
\begin{gather*}
(t(\xbar) \not\equiv_{m,\theta} 0 \wedge t(\xbar) \equiv^{[m]}_{m,\theta} 0)
\,\,\vee\\
 (\theta \text{ is the minimal element of }\Auxc_m \wedge t(\xbar) \equiv_{m,\theta} 0)
\end{gather*}
(the second line treats the case $t(\xbar) \equiv_{m} 0$),
and $\cane_p(t(\xbar)) = \theta$ is equivalent to
\begin{gather*}
(t(\xbar) \ne_{\theta} 0 \wedge t(\xbar) =_{\theta+} 0)
\,\,\vee\\
(\theta \text{ is the minimal element of }\Auxe_p \wedge t(\xbar) = 0)
,
\end{gather*}
where $t(\xbar) =_{\theta+} 0$ can be written in family union form as
\[
\exists (\theta' \in \Auxe^+_p)\,(\theta' = \theta\mathord{+} \wedge t(\xbar) =_{\theta'} 0).
\]
(Here, we use $\Main$-qf definability of $\theta \mapsto \theta+$).

Now our formula $\phi(\xbar, \etabar)$ is purely in the language $\Lint$ and it is of the form
$\exists \thetabar\, \psi(\xbar, \etabar, \thetabar)$,
where $\thetabar$ is auxiliary and $\psi$ is a boolean combination of quantifier free
parts and of parts living purely in $\Aux$.
Moreover, by the way in which the quantifier $\exists \thetabar$ has been introduced,
$\psi(\xbar, \etabar, \alphabar)$ and
$\psi(\xbar, \etabar, \alphabar')$ are inconsistent
for any $\alphabar, \alphabar' \in \Aux$,
$\alphabar \ne \alphabar'$. Thus to turn $\phi(\xbar, \etabar)$ into
family union form, it remains to bring $\psi$ into a disjunctive normal form where the conjunctive clauses are pairwise inconsistent, and then pull the disjunction to the outside (here, we treat the $\Aux$-parts of $\psi$ with quantifiers as atoms).
This kind of disjunctive normal form can be obtained by using conjunctive clauses each
of which contains all atoms occurring in $\psi$, either positively or negatively.
\end{proof}

Now it is easy to deduce $\Lint$ quantifier elimination from $\Lsyn$ quantifier elimination:

\begin{proof}[Proof of Theorem~\ref{thm:qe-int} from Theorem~\ref{thm:qe-form}]
Any $\Lint$-formula is equivalent to an $\Lsyn$-formula. Using Theorem~\ref{thm:qe-form},
we can turn this into an $\Lsyn$-formula in family union form
\[
\phi(\xbar, \etabar) =
\bigvee_{i = 1}^k
\exists \thetabar\, \big(\xi_i(\etabar, \thetabar) \wedge \psi_i(\xbar, \thetabar)\big)
.
\]
Since $\Lsyn$ and $\Lint$ agree on the auxiliary sorts, the formulas $\xi_i$
are also $\Lint$-formulas.
By Proposition~\ref{prop:L2-to-L3}, we may replace each $\psi_i$ by
an $\Lint$-formula in family union form. By pulling the quantifiers and disjunctions
of these $\psi_i$ to the outside, we obtain a formula which is in family union
form as a whole.
\end{proof}

\section{The main proofs}
\label{sect:proofs}

\subsection{Partial quantifier elimination in general}
\label{subsect:no-groups}

In this section, we prove Proposition~\ref{prop:no-groups} which gives
a general method to eliminate main sort quantifiers
when the only connection between the main sorts and the auxiliary sorts
are functions from $\Main$ and $\Aux$.
The proof goes in two steps;
we formulate the first one as a separate lemma.

\begin{lem}\label{lem:no-groups}
Let $L$ be a language, let $\Main \dcup \Aux$ be a partition of the sorts of $L$, and
suppose that the only symbols in $L$ connecting $\Main$ and $\Aux$
are functions from (products of) $\Main$ sorts to $\Aux$ sorts.
Then any formula without $\Main$-quantifiers is equivalent to
a formula in family union form (in any theory).
\end{lem}

\begin{proof}
Let $\phi$ be an $\Main$-qf formula.
We do an induction over the number of occurrences of main variables in
$\phi$. If no main variable appears in $\phi$, there is nothing to do.
Otherwise, choose a specific occurrence of a main variable $x$ in $\phi$.
We distinguish the following two cases:

(1) The atom $a$ containing $x$ is a relation on $\Main$ (applied to some terms
living completely in $\Main$).

(2) $x$ appears inside a term $t$ with range in $\Aux$.

In case (1), every other variable appearing in the atom $a$ is also
a main sort variable, so
$a$ does not depend on any of the quantified
variables of $\phi$, and we can ``do a case distinction on $a$'':
$\phi$ is equivalent to
\[
(a \wedge \phi[\tfrac\top a]) \vee (\neg a \wedge \phi[\tfrac\bot a])
.
\]
(Here, the notation $\phi[\frac rs]$ means: the formula obtained from $\phi$
by replacing all occurrences of $s$ by $r$, 
$\top$ means true and $\bot$ means false.)
Apply the induction hypothesis to $\phi[\frac\top a]$
and $\phi[\frac\bot a]$. After pulling the ``$a \wedge\strut$'' and
``$\neg a \wedge\strut$'' inside,
the result is in family union form.

In case (2), consider the smallest subterm $t'$ of $t$
containing $x$ whose range lies in $\Aux$. Then the outermost function of $t'$ is
a function from a product of some $\Main$-sorts to $\Aux$,
so $t'$ depends only on $\Main$-variables and in particular not on
quantified variables. Thus $\phi$ is equivalent to
\[
\exists \xi\,(t' = \xi \wedge \phi[\tfrac\xi{t'}])
.
\]
Applying induction to $\phi[\frac\xi{t'}]$ yields a formula
in family union form.
\end{proof}

Note that this lemma in particular implies that the negation of
a formula in family union form can again be brought into
family union form.

Now let us get to the main proof of this section:

\begin{proof}[Proof of Proposition~\ref{prop:no-groups}]
Let $\phi$ be a formula whose $\Main$-quantifiers we want to eliminate.
We use induction over the structure of $\phi$, i.e., we suppose that the
subformulas are already in family union form. By
Lemma~\ref{lem:no-groups}, it suffices to bring $\phi$ into a
form without $\Main$-quantifiers.

If $\phi$ is an atom, then there is nothing to do, and neither if
it is of the form $\neg \psi$ or $\psi_1 \wedge \psi_2$, so suppose
$\phi = \exists x\,\psi(x)$, where $x$ is a main sort variable and $\psi(x)$ is in family union form, i.e.:
\[
\phi(\ybar, \etabar) = \exists x\,
\bigvee_{i = 1}^k
\exists \thetabar\, \big(\xi_i(\etabar, \thetabar) \wedge \psi_i(x, \ybar, \thetabar)\big)
\]
Rewrite this as
\[
\bigvee_{i = 1}^k
\exists \thetabar\, \big(\xi_i(\etabar, \thetabar) \wedge \exists x\,\psi_i(x, \ybar, \thetabar)\big)
.
\]
Since $\psi_i(x, \ybar, \thetabar)$ is quantifier free,
the hypothesis of the proposition applies to $\exists x\,\psi_i(x, \ybar, \thetabar)$, and we get a formula without main sort quantifiers.
\end{proof}

\subsection{Removing the quantifier in $X + G'$}

At one point in the main proof of quantifier elimination, we will have
a subgroup $G' \subseteq G$ and
a set $X \subseteq G$ defined by a quantifier free formula of a particular form
and we will need to be able to define the set $X + G'$ without quantifiers.
This will be possible using the following two lemmas
which have nothing to do with model theory.

\begin{lem}\label{lem:sat}
Suppose we have an abelian group $G$, a subgroup $G' \subseteq G$ and a subset $X \subseteq G$
of the form
\[
X = (H_0 + a_0) \setminus \bigcup_{i = 1}^\nu (H_i + a_i)
\]
where $H_i$ are subgroups of $G$, $a_i \in G$, and where
$H_i + a_i \subseteq H_0 + a_0$ for $i\in\{1,\dots, \nu\}$ and $H_i + a_i \cap H_j + a_j = \emptyset$ for
$i,j \in\{1,\dots, \nu\}, i \ne j$.
Then for $x \in G$ we have $x \in X' := X + G'$ if and only if
\begin{gather*}
\tag{1}
x - a_0 \in H_0 + G' \quad\text{and}\\
\tag{2}
\sum_{\{1\le i \le \nu \mid x - a_i \in H_i + G'\}}
\hskip-1ex
((H_0 \cap G') : (H_i \cap G'))^{-1} < 1
.
\end{gather*}
(Here, we use the convention $\infty^{-1} = 0$.)
\end{lem}

\begin{proof}
The condition $b \in X'$ is equivalent to $X \cap (b + G') \ne \emptyset$.
Write
\[
X \cap (b + G') = C_0 \setminus \bigcup_{i = 1}^\nu C_i,
\]
with $C_i := (a_i + H_i) \cap (b + G')$.
Then $C_i$ is non-empty if and only if $b - a_i \in H_i \cap G'$,
and if it is non-empty, then it is of the form
$c_i + H_i \cap G'$.

Non-emptiness of $C_0$ is just condition~(1) on $b$ in the lemma,
so suppose now that $C_0$ indeed is non-empty. The question is now whether
the union $\bigcup_{i = 1}^\nu C_i$ (which is disjoint) contains all of $C_0$.
The sum in condition~(2)
goes exactly over those $i \ge 1$ for which $C_i$ is non-empty, and the
summand is the proportion of $C_i$ in $C_0$. Hence
$\bigcup_{i = 1}^\nu C_i = C_0$ if and only if the sum is $1$.
(To make this more formal, count elements in
$C_0/D$, where $D$ is the intersection of all those $C_i$ which have
finite index in $C_0$.)
\end{proof}

The next lemma will be helpful to make condition (2) from the previous lemma definable.

\begin{lem}\label{lem:power-sum}
Suppose that $n \in \NN$, $n \ge 2$ and that $q_1, \dots, q_\nu$ are powers of $n$.
Then there exists an $N \in \NN$ depending only on $n$ and $\nu$
such that
\[
\sum_{i=1,\dots, \nu} q_i^{-1} \ge 1
\quad\iff\quad
\sum_{\substack{i=1,\dots, \nu\\q_i < n^N}} q_i^{-1} \ge 1
.
\]
\end{lem}

\begin{proof}
Choose $N$ such that $\nu < N \cdot (n - 1) + 1$.
Without loss, $q_1 \le \dots \le q_\nu$.
Set $s_k := \sum_{i = 1}^{k} q_i^{-1}$ and let
$d_k$ be the digit sum of $s_k$ in base $n$.
Inductively, one proves $d_k \le k$.
If the claim of the lemma is false, then there exists an $\ell \le \nu$  with
$q_\ell \ge n^N$ such that $s_{\ell-1} < 1 \le s_{\ell}$.
This implies $d_{\ell-1} \ge N\cdot (n-1)$,
contradicting $d_{\ell-1} \le \ell - 1 \le \nu - 1 < N \cdot (n - 1)$.
\end{proof}

\subsection{Actually eliminating the quantifiers}

\begin{proof}[Proof of Theorem~\ref{thm:qe-form}]
%\begin{proof}[Proof of Proposition~\ref{prop:index-aware}]

As announced, we prove Theorem~\ref{thm:qe-form} using Proposition~\ref{prop:no-groups},
i.e., we have to show that if $\phi(x, \ybar, \etabar)$
is a quantifier free $\Lsyn$-formula, then
$\exists x\,\phi(x, \ybar, \etabar)$ is equivalent to an $\Main$-qf $\Lsyn$-formula.
Since the language $\Lint$ is more intuitive, we start by translating $\phi$
into an $\Lint$-formula using Proposition~\ref{prop:L2-to-L3}.
The result is in family union form, i.e., we have to eliminate ``$\exists x$'' from
a formula of the form
\[
\exists x\,
\bigvee_{i = 1}^k
\exists \thetabar\, \big(\xi_i(\etabar, \thetabar) \wedge \psi_i(x, \ybar, \thetabar)\big)
.
\]
By pulling this quantifier inside, it suffices to eliminate the quantifier
of $\exists x\,\psi_i(x,\ybar,\thetabar)$.
In other words, we now need to eliminate the quantifier of
$\exists x\,\phi(x, \ybar, \etabar)$ when
$\phi(x,\ybar, \etabar)$ of the form
\[\tag{*}
\phi(x,\ybar, \etabar) =
   \bigwedge_{i = 1}^k r_i x \mathrel{(\diamrel_i)}_{\eta_i} y_i + k_{\eta_i}
\]
with $r_i \in \NN$, $\diamond_i \in
\{=,\ne,<,>,\le, \ge,\equiv_m,\not\equiv_m,\equiv^{[n]}_m,\not\equiv^{[n]}_m\}$.

We will show that $\exists x\,\phi(x, \ybar, \etabar)$ is equivalent to
an $\Main$-qf formula in the language $\Lsyn \cup \Lint$; this is enough, since
afterwards, we can apply Proposition~\ref{prop:L3-to-L2} to translate the
$\Lint$-predicates into $\Lsyn$.

To simplify the exposition, let us choose
parameters $\bbar \in \Main$, $\alphabar \in \Aux$
and consider $\phi(x, \bbar, \alphabar)$; we will denote this by $\phi(x)$
for short. Our strategy is to successively simplify $\phi(x)$;
of course, the whole point is that this is done in a way depending definably on
the parameters $\bbar$ and $\alphabar$.

If $x \diamrel_\alpha b + k_\alpha$ is a literal of $\phi$, we will
write $b_*$ for a representative in $G$ of $b + k_\alpha$,
so that $x \diamrel_\alpha b + k_\alpha$ is equivalent to
$x \diamrel_\alpha b_*$; if $k = 0$ or $G/G_\alpha$ is dense, we set $b_* := b$.
We will sometimes use $x \diamrel_\alpha b_*$
as a short hand notation for $x \diamrel_\alpha b + k_\alpha$.
Of course, we are not allowed to use $b_*$ in the
resulting quantifier free formula. However, if we have an element $\alpha' \ge \alpha$,
then a condition of the form $t\diamrel'_{\alpha'} b_*$ can easily
be expressed using $b$ and $k$ instead of $b_*$:
\[
t\diamrel'_{\alpha'} b_*
\quad\iff\quad
(\alpha' = \alpha \wedge t \diamrel'_{\alpha} b + k_\alpha)
\,\,\vee\,\,
(\alpha' > \alpha \wedge t \diamrel'_{\alpha'} b)
;
\]
we will use this without further mentioning.

Now let us get to work.
First we get rid of the factors $r_i$ in (*). To this end, note that
in each literal, multiplying both sides by any non-zero integer $r$ doesn't
change the set defined by that literal if additionally we do the following:
\begin{itemize}
\item
in literals with $\equiv_m$, $\not\equiv_m$, $\equiv^{[n]}_m$, $\not\equiv^{[n]}_m$, we also
multiply $m$ (and $n$) by $r$ (this uses Lemma~\ref{lem:cong-mult});
\item
we turn inequalities around if $r < 0$.
\end{itemize}
In this way, we can make all $r_i$ equal to one single $r$.
After that, we replace $rx$ by a new variable $x'$ and replace $\exists x$
by $\exists (x' \in rG)$.

The remainder of the proof will consist of two big parts:
in the first one, we get rid of the inequalities ($\ne, <, >, \le, \ge$);
in the second one, we treat the congruence conditions
($\equiv_m,\not\equiv_m,\equiv^{[n]}_m,\not\equiv^{[n]}_m$).

\medskip
\textbf{Part 1: treating inequalities}

Our goal in this part is to
reduce the quantifier elimination problem to formulas $\phi(x)$ of the form
\[
\tag{**}
\phi'(x)
\quad\text{or}\quad
x =_\delta b_* \wedge \phi'(x),
\]
where the atoms of $\phi'(x)$ use only $\equiv_m$ and $\equiv^{[n]}_m$.

We start by replacing literals of the form $x =_\alpha b_*$ and $x \ne_\alpha b_*$
by $x \ge_\alpha b_* \wedge x \le_\alpha b_*$
and $x >_\alpha b_* \vee x <_\alpha b_*$, respectively. In the second case,
we treat each disjunct separately. (Replacing $x =_\alpha b_*$ by inequalities
might seem strange at first sight, since later, we want to get back to equalities.
However, recall that after all, $x =_\alpha b_*$ defines an interval.)

Next, reduce to the case where $\phi(x)$ contains at most one lower and one upper bound:
if $\phi = \phi'' \wedge \psi_1 \wedge \psi_2$, where $\psi_i$ are two bounds
on the same side, then $\exists x \,\phi$ is equivalent to
$(\exists x\, (\phi'' \wedge \psi_1)) \wedge (\exists x\, (\phi'' \wedge \psi_2))$.
Thus $\phi(x)$ is now of the form
\[
\tag{***}
\phi(x) =\quad
c_* \vartriangleleft_\alpha x \vartriangleleft'_{\alpha'} c'_*
\,\wedge\, \phi'(x),
\]
where $\mathord{\vartriangleleft}, \mathord{\vartriangleleft'} \in \{\le, <, \text{no condition}\}$ and
where the atoms of $\phi'(x)$ use only $\equiv_m$ or $\equiv^{[n]}_m$.
Such an atom defines a union of cosets of $mG$, hence if we let $m_0$ be the least common multiple of all occurring $m$, then $\phi'(G)$ is a union of cosets of $m_0G$.
We fix this $m_0$ for the remainder of the proof.

If $\phi(x)$ has no bounds, then it is already of the form (**).
If $\phi(x)$ has only one bound, say, a lower one, then removing that bound does not
change the truth of $\exists x\,\phi(x)$. Indeed, if an element $a \in G$ satisfies $\phi'(x)$
but does not satisfy $c_* \vartriangleleft_\alpha x$, all elements of
$a + m_0G$ satisfy $\phi'(x)$ and in that set, we can find one which also satisfies the bound. Hence for the remainder of part~1, we assume that $\phi(x)$ has two bounds.

If $\alpha \ge \alpha'$ in (***), we may suppose that $c_* \vartriangleleft_\alpha c'_*$,
since otherwise $\phi(x)$ defines the empty set. Similarly, if
$\alpha' \ge \alpha$, we may suppose that $c_* \vartriangleleft'_{\alpha'} c'_*$.

Let $\gamma$ be an auxiliary element satisfying
$\gamma \asymp \max\{\alpha, \alpha', \cane_{m_0}(c_* - c'_*)\}$.
Recall that $c_*$ is a representative of $c + k_\alpha$ (for some $c \in G$ and $k \in \ZZ$),
but since $\cane_{m_0}(c - c_*) \le \alpha$, $\gamma$ does not depend on the choice of $c_*$
(and similarly for $c'_*$).
By Lemma~\ref{lem:prime-e}, $\gamma$ is definable. (Formally, it is an element of
one of finitely many auxiliary sorts; we do a case distinction on the sort.)

Suppose that $\alpha < \gamma$. We claim that then, weakening the lower bound
from $c_* \vartriangleleft_\alpha x$ to $c_* \le_\gamma x$ does not change the truth
value of the formula $\exists x\,\phi(x)$.
In other words, we claim that if there exists an
$a \in G$ with $a =_\gamma c_* \,\wedge\, a \vartriangleleft'_{\alpha'} c'_*
\,\wedge\, \phi'(a)$, then we can
find an $a'$ satisfying $\phi(x)$.
If $c_* \vartriangleleft_\alpha a$, there is nothing to do.
Otherwise, we can choose an element $a_0 \in m_0G_\gamma$ such that
$c_* \vartriangleleft_\alpha a + a_0 =: a'$.
By construction, $a'$ satisfies $\phi'(x)$ and the lower bound.
Concerning the upper bound: if
$\gamma = \alpha'$, then $a' =_\gamma c_* \vartriangleleft'_{\alpha'} c'_*$
implies $a' \vartriangleleft'_{\alpha'} c'_*$.
If, on the other hand, $\gamma > \alpha'$, then by definition of $\gamma$ we have
$\gamma = \cane_{m_0}(c_* - c'_*)$, hence $c_* \ne_\gamma c'_*$,
and hence $a' =_\gamma c_* <_\gamma c'_*$,
which again implies $a' \vartriangleleft'_{\alpha'} c'_*$.

We do the same with the upper bound and thus get a formula of the form
\[
c_* \vartriangleleft_\gamma x \vartriangleleft'_{\gamma} c'_*
\wedge \phi'(x),
\]
where $\gamma \ge \cane_{m_0}(c_* - c'_*)$.

Now we distinguish two cases, depending on whether
$c'_* - c_* >_\gamma (m_0+1)_\gamma$ or not. (Recall that if $G/G_\gamma$ is dense, then
by definition this is equivalent to $c'_* - c_* >_\gamma 0$.)

Suppose first that the condition is false.
If $G/G_\gamma$ is dense, then this implies $c_* =_\gamma c'_*$, so
$\phi(x)$ can only
be consistent if both inequalities are non-strict, and in that case,
it is equivalent to $x=_\gamma c_* \wedge \phi'(x)$, which is of the form (**).
If $G/G_\gamma$ is discrete, then $c'_* - c_* =_\gamma \ell_\gamma$
for some $\ell \le m_0+1$.
Thus
$\phi(x)$ is equivalent to the disjunction of
finitely many formulas of the form
\[
x =_\gamma c_* + i_\gamma \wedge \phi'(x)
.
\]
More precisely, $i$ runs from $0$ or $1$ (depending on $\vartriangleleft$)
to $\ell - 1$ or $\ell$ (depending on $\vartriangleleft'$).

Now suppose that $c'_* - c_* >_\gamma (m_0 + 1)_\gamma$.
Then there exists an element $d \in G$ satisfying
$0 <_{\gamma} (m_0 + 1)d <_{\gamma} c'_* - c_*$.
(If $G/G_\gamma$ is discrete, then choose for $d$ any representative
of $1_\gamma$.)
%Now suppose that there exists a $d \in G$ such that $0 <_{\gamma} m_0d <_{\gamma} c' - c$.
Using this, we will show that $\exists x\,\phi(x)$ is equivalent to
\[
\exists x\, (x =_{\delta} c_*\wedge \phi'(x)),
\]
where $\delta := (\cane_{m_0}(c'_* - c_*))+$.
(Again, $\delta$ is definable by Lemma~\ref{lem:prime-e}.)

It is clear that $\phi(x)$ implies
$x =_{\delta} c_*$, since $c'_* =_{\delta} c_*$, so it remains to
show that if there exists an $a \in c_* + G_{\delta}$ satisfying $\phi'(x)$,
then there exists an $a' \in G$ which additionally lies between the bounds.

The inequality $\canc_{m_0}(a - c_*) \le \cane_{m_0}(a - c_*) \le \cane_{m_0}(c_* - c'_*) \le \gamma$
means that for any convex subgroup $H \subseteq G$ strictly containing $G_\gamma$, we have
$a - c_* \in H + m_0G$. In particular, since $d \ne_{\gamma} 0$,
\[
a - c_* \in \langle d\rangle\conv + m_0G = [d, (m_0 + 1)d] + m_0G
.
\]
Choose $a_0 \in (a - c_* + m_0G) \cap [d, (m_0 + 1)d]$ and set $a' := a_0 + c_*$.
Then $a'$ satisfies $\phi'(x)$ since it differs from $a$ by an element of $m_0G$,
and $0 <_\gamma d \le a_0 \le (m_0 + 1)d <_\gamma c'_*-c_*$ implies that $a'$ also
satisfies the bounds.

\medskip
\textbf{Part 2: treating congruences}

Our formula $\phi(x)$ is now of the form
\[
\phi'(x) \quad \text{or} \quad x =_\gamma c_* \wedge \phi'(x)
,
\]
where the atoms of $\phi'(x)$ are of the form
$x \equiv_{m,\alpha} b_*$ or $x \equiv^{[n]}_{m,\alpha} b_*$.
Using Lemma~\ref{lem:cong-factorize}, we can suppose that each $m$ and each $n$ is a power of a prime, and using Lemma~\ref{lem:m-mid-n}, we get rid of all those atoms
$x \equiv^{[n]}_{m,\alpha} b$ where $m$ and $n$ are powers of different primes.
By the Chinese remainder theorem, we can eliminate the quantifier separately
for each of the subformulas of $\phi'$ corresponding to the different primes.
In other words, we may assume that all atoms of $\phi'$ are of the form
$x \equiv_{p^r,\alpha} b_*$ or $x \equiv^{[p^s]}_{p^r,\alpha} b_*$ for one
single prime $p$ which we fix for the remainder of the proof.
Moreover, in $\equiv^{[p^s]}_{p^r,\alpha}$ we may assume $s \ge r$ (again by
Lemma~\ref{lem:m-mid-n}).

From now on, we also fix $r$ to be the maximal
exponent of $p$ appearing in the atoms in the above way (both, in $\equiv_{p^r,\alpha}$ and in $\equiv^{[p^s]}_{p^r,\alpha}$); in particular, $\phi'(G)$ consists
of entire cosets of $p^rG$.

In general, if $\phi(x) = \phi_0(x) \wedge \phi_1(x)$ and $H \subseteq G$ is any subgroup such that
$\phi_1(G)$ consists of entire cosets of $H$,
then replacing $\phi_0$ by a formula defining $\phi_0(G) + H$ does not change
the truth of $\exists x\,\phi(x)$; we will apply this enlargement argument several times.
Since $\phi'(G)$ is a union of cosets of $p^rG$, we already 
can replace $x=_\gamma c_*$ by $x \equiv_{p^r,\gamma} c_*$, i.e., without loss
there is no literal $x=_\gamma c_*$.

Now we prove quantifier elimination by induction on $r$.
If $r = 0$, then $\exists x\,\phi(x)$ is equivalent to $\phi(0)$.
For the induction step, suppose $r > 0$ and
write $\phi = \phi_0 \wedge \phi_1$, where $\phi_0$ contains the atoms
$x \equiv_{m,\alpha} b_*$, $x \equiv^{[n]}_{m,\alpha} b_*$ with
$m = p^r$ and $\phi_1$ contains the atoms with $m \le p^{r-1}$.
By the enlargement argument, we are done with the induction step
if we can show the following:
\begin{itemize}
\item[(a)]
the set $\phi_0(G) + p^{r-1}G$
is definable by a formula $\phi'_0$ using only atoms of the form
$x \equiv_{p^{r-1},\alpha} b_*$, $x \equiv^{[s]}_{p^{r-1},\alpha} b_*$,
with $r$ as given and $s \ge r - 1$ arbitrary;
\item[(b)]
$\phi'_0$ depends on the parameters of $\phi_0$ in an $\Main$-qf definable way.
\end{itemize}

The atoms
$x \equiv_{p^r,\alpha} b_*$, $x \equiv^{[p^s]}_{p^r,\alpha} b_*$ of $\phi_0$
define cosets of groups, and these groups are totally ordered by inclusion:
\[
\dots\subseteq
G_{\alpha} + p^rG \subseteq \dots \subseteq
G^{[p^{s+1}]}_{\alpha} + p^rG \subseteq G^{[p^s]}_{\alpha} + p^rG
\subseteq \dots \subseteq G_{\alpha'} + p^rG
\subseteq\dots
\]
for all $\alpha < \alpha'$ and all $s \ge r$.
In particular, any two such cosets $H + b_*$, $H' + b'_*$
are either disjoint or contained
in one another. Moreover, whether $H + b_* \subseteq H' + b'_*$ or not is definable.
Using this, we can simplify $\phi_0$ such that
it has at most one positive literal, and all negative literals
exclude pairwise disjoint sets. Now $\phi(G)$ satisfies the prerequisites of
Lemma~\ref{lem:sat}:
in that lemma, let $H_0 + a_0$ be the set defined by the positive literal
of $\phi_0$ (or $H_0 =G, a_0 = 0$ if there is no positive literal),
let $H_i + a_i$ be the sets excluded by the negative literals,
and set $G' := p^{r-1}G$. (The $a_i$ are the representatives denoted by $b_*$ before.)
To get our desired formula defining $X' = \phi_0(G) + p^{r-1}G$,
it remains to verify that conditions (1) and (2)
of Lemma~\ref{lem:sat} are $\Main$-qf definable, where $x$ only appears in atoms as in (a).

For each $i\in \{0,\dots, \nu\}$
we have
\[
H_i = G
\quad\text{or}\quad
H_i = G_{\alpha} + p^rG
\quad\text{or}\quad
H_i = G^{[p^s]}_{\alpha} + p^rG
,\]
so the condition $x - a_i \in H_i + p^{r-1}G$
is definable by
\[
x = x
\quad\text{or}\quad
x \equiv_{p^{r-1},\alpha} a_i
\quad\text{or}\quad
x \equiv^{[p^s]}_{p^{r-1},\alpha} a_i
.\]
This settles definability of (1), and it allows us to do a case distinction which fixes the set
the sum (2) runs over. Let $I$ be that set and set $q_i := \big((H_0 \cap p^{r-1}G) : (H_i \cap p^{r-1}G))$ for $i \in I$.
By Lemma~\ref{lem:cong-mult}, for $i \in I \cup \{0\}$ we can write $H_i \cap p^{r-1}G$ as $p^{r-1}H_i'$ with
\[
H'_i = G
\quad\text{or}\quad
H'_i = G_{\alpha} + pG\quad\text{or}\quad
H'_i = G^{[p^{s-r+1}]}_{\alpha} + pG
,\]
so $q_i = (H'_0 : H'_i)$ is the cardinality of a quotient treated by
Lemma~\ref{lem:dim-def}. Thus each $q_i$ is either infinite or a power of $p$,
and the conditions $q_i = p^\ell$ (for $\ell \in \NN_0$) are $\Main$-qf definable.
By Lemma~\ref{lem:power-sum}, there exists a bound $N$
such that $\sum_{i \in I}q_i < 1$ iff $\sum_{i \in I,q_i < p^N}q_i < 1$.
The latter is equivalent to a finite boolean combination of conditions
of the form $q_i = p^\ell$ (for $i \in I$ and $\ell < N$).
Hence, condition (2) is definable, too, and we are done.
\end{proof}

\section{Examples}
\label{sect:ex}

In this section, we give some examples which should help the reader understanding
the languages which we define. These examples show that large parts of the
languages is indeed necessary. More detailed examples explicitly concerning $\Lint$ are given in \cite{i.oAGlang}; similar motivating examples, but presented from a different point of view
are given in \cite{Sch.oAGmodComp}.

\subsection{Concrete examples illustrating the sort $\Auxc_n$}
\label{subsect:exc}

Set $G = \ZZ \oplus \ZZ$ with lexicographical order.
We determine the sort $\Auxc_n$ for $n \ge 2$. For this, we have to go through all
elements $a \in G \setminus nG$ and find the largest convex
subgroups $H = G_{\canc_n(a)} \subseteq G$ such that $H + nG$ does not contain $a$.
Equivalently, $H$ is the largest convex
subgroup which is disjoint from $a + nG$.

Obviousely, $H$ only depends on the class of $a$ modulo $nG$.
%For $a \in nG$, we defined $G_{\canc_n(a)}$ to be $\{(0,0)\} =: G_0$.
If $a = (0,z)$ for $z \notin n\ZZ$, then we have $H = \{(0,0)\} =: G_0$;
if $a = (z,z')$ for $z \notin n\ZZ$ and
$z' \in \ZZ$ arbitrary, then $H = \{0\} \times \ZZ =: G_1$.
Thus $\Auxc_n$ consists of two elements which correspond
to the groups $G_0$ and $G_1$. (For $a \in nG$, by definition we also
have $G_{\canc_n(a)} = G_0$.)

In this example, all sorts $\Auxc_n$ are the same.
Now consider the group $G = \ZZ[\frac15] \oplus \ZZ$ instead.
The sorts $\Auxc_n$ for $n \ne 5^r$ are the same as before; however, the sort $\Auxc_{5^r}$ now consists of a single element, since modulo $5^rG$, any element of $G$ is equivalent
to an element of the form $(0,z)$.

In these examples, the sorts $\Auxe_n$ and $\Auxe^+_n$ do not yield any new
non-trivial convex subgroups of $G$:
$G_{\cane_n(a)}$ is $G_0$ if $a \in G_1$ and $G_1$ otherwise, and
$G_{\cane_n(a)+}$ is $G_1$ if $a \in G_1$ and $G$ otherwise.
To get interesting new convex subgroups, we have to consider infinite lexicographical
products.

\subsection{Infinite lexicographical products illustrating $\Auxe_n$ and $\Auxe^+_n$}
\label{subsect:exe}

Let $I$ be any ordered set, and
let $G := \bigoplus_{i \in I} \ZZ$ be the group with lexicographical order
``with significance according to $I$''. More precisely,
for $a = (a_i)_{i \in I} \in G$, set $v(a) := \max\{i \in I \mid a_i \ne 0\}$ if $a \ne 0$ and
$v(0) := -\infty$. (This is well-defined, since only finitely many $a_i$ are non-zero.)
Now define the order on $G$ by $a > 0$ iff $a \ne 0$ and $a_{v(a)} > 0$.

For $j \in I$, let us write $g_j$ for the map $\ZZ \to G$ sending $\ZZ$ to the $j$-th summand
of $G$. Now let us determine $\Auxc_n$ (for $n \ge 2$).
For $j \in I$, the largest convex subgroup not intersecting $g_j(1) + nG$
is $H_{<j} := \{g \in G \mid v(g) < j\}$, thus we get an injection $I \inject \Auxc_n$.
For arbitrary $a = (a_i)_{i \in I} \in G \setminus nG$, we do not get more groups:
$G_{\canc_n(a)}$ is equal to $H_j$, where $j \in I$ is the largest index such that $a_j \notin n\ZZ$.
Thus $\Auxc_n$ is equal to $I$, possibly enlarged by one element
corresponding to the group $\{0\}$ (since $G_{\canc_n(a)} = \{0\}$ for
$a \in nG$). In particular, $I$ can be interpreted in $G$.

Now consider the sorts $\Auxe_n$ and $\Auxe^+_n$.
The group $G_{\cane_n(a)}$ is the union of all $H_{<j}$ not containing $a$, so it is
equal to $H_{<v(a)}$; still nothing new. However, $G_{\cane_n(a)+}$ is the intersection
of all $H_{<j}$ containing $a$, i.e. $G_{\cane_n(a)+} = H_{\le v(a)} := \{g \in G \mid v(g) \le v(a)\}$ which
might be a group which we did not have before.

Now modify our example by choosing a subset $I' \subseteq I$ and by replacing, for each
$j \in I'$, the factor $\ZZ$ of $G$ by $\QQ$.
Then, $\Auxc_n$ parametrizes only those groups $H_{< j}$
for which $j \in I \setminus I'$. However, elements $a \in G$ with $v(a) \in I'$
can still be used to obtain elements of $\Auxe_n$ and $\Auxe^+_n$; thus now all
three sorts can be really different.
To give an extreme example, take $I = \RR$ and $I' = \RR \setminus \QQ$;
then, as ordered sets, we have $\Auxc_n \cong \{-\infty\} \dcup \QQ$,
whereas $\Auxe_n \cong \Auxe^+_n \cong \{-\infty\} \dcup \RR$.

\subsection{An example for $G^{[n]}_\alpha$}
\label{subsect:ex[n]}

In general, the group
\[\tag{*}
H_1 := G^{[n]}_\alpha = \bigcap_{H \supsetneq G_{\alpha}} (H + nG)
\]
(where $H$ runs over convex subgroups of $G$)
is not of the form $H_0 + nG$ for any convex subgroup $H_0$ of $G$.
Here is an example. We use the notation from Section~\ref{subsect:exe}.
Let $I = \NN$,
but with reversed order; set $G' := \bigoplus_{i \in I} \ZZ$ (ordered
as in Section~\ref{subsect:exe}),
and let $G$ be the subgroup of $G'$ consisting of those $(a_i)_{i\in I} \in G'$
with $\sum_i a_i \in n\ZZ$ (for any fixed $n \ge 2$).

Choose $\alpha := \canc_n(0)$ and define $H_1$ by (*). Then $G_\alpha = \{0\}$, and
the largest convex subgroup of $G$ contained in $H_1$ is $\{0\}$, so
the only candidate of the form $H_0 + nG$ which could be equal to $H_1$ is
$nG$ itself.

Any element $(a_i)_{i\in I} \in nG$ satisfies $\sum_i a_i \in n^2\ZZ$.
On the other hand, for any non-trivial convex subgroup $H \subseteq G$, we have
$H + nG = H + nG'$, since the condition $\sum_i a_i \in n\ZZ$ can always be satisfied by adding an element of $H$. Thus $H_1 = nG'$, which is strictly bigger than $nG$.

\bibliographystyle{amsplain}

\begin{thebibliography}{10}

\bibitem{BVW.ordGrp}
Oleg Belegradek, Viktor Verbovskiy, and Frank~O. Wagner, \emph{Coset-minimal
  groups}, Ann. Pure Appl. Logic \textbf{121} (2003), no.~2-3, 113--143.
  \MR{MR1982944 (2004i:03062)}

\bibitem{iC.lipQp}
Raf Cluckers and Immanuel Halupczok, \emph{Approximations and lipschitz
  continuity in p-adic semi-algebraic and subanalytic geometry}, 2010,
  Preprint.

\bibitem{Gur.decProb}
Y.~Gurevic, \emph{The decision problem for some algebraic theories}, 1968,
  Doctor of Mathematics dissertation.

\bibitem{GS.oAGNIP}
Y.~Gurevich and Peter~H. Schmitt, \emph{The theory of ordered abelian groups
  does not have the independence property}, Trans. Amer. Math. Soc.
  \textbf{284} (1984), no.~1, 171--182. \MR{MR742419 (85g:03053)}

\bibitem{Gur.oAGelProp}
Yu.Sh. Gurevich, \emph{{Elementary properties of ordered Abelian groups.}}, Am.
  Math. Soc., Transl., II. Ser. \textbf{46} (1964), 165--192 (English. Russian
  original).

\bibitem{i.oAGlang}
Immanuel Halupczok, \emph{A language for quantifier elimination in ordered
  abelian groups}, S{\'e}minaire de Structures Alg{\'e}briques Ordonn{\'e}es
  2009--2010 (Fran{\c{c}}oise Delon, Max~A. Dickmann, and D.~Gondard, eds.),
  vol.~85, {\'E}quipe de Logique Math{\'e}matique, Paris, 2011.

\bibitem{Poi.modTh}
Bruno Poizat, \emph{Cours de th\'eorie des mod\`eles}, Bruno Poizat, Lyon,
  1985, Une introduction {\`a} la logique math{\'e}matique contemporaine. [An
  introduction to contemporary mathematical logic]. \MR{817208 (87f:03084)}

\bibitem{Rub.linOrd}
Matatyahu Rubin, \emph{Theories of linear order}, Israel J. Math. \textbf{17}
  (1974), 392--443. \MR{0349377 (50 \#1871)}

\bibitem{Sch.habil}
P.~H. Schmitt, \emph{Model theory of ordered abelian groups}, 1982,
  Habilitationsschrift.

\bibitem{Sch.oAGmodComp}
Peter~H. Schmitt, \emph{Model- and substructure-complete theories of ordered
  abelian groups}, Models and sets ({A}achen, 1983), Lecture Notes in Math.,
  vol. 1103, Springer, Berlin, 1984, pp.~389--418. \MR{MR775703 (86h:03063)}

\bibitem{Wei.ordAbGrp}
Volker Weispfenning, \emph{Elimination of quantifiers for certain ordered and
  lattice-ordered abelian groups}, Proceedings of the {M}odel {T}heory
  {M}eeting ({U}niv. {B}russels, {B}russels/{U}niv. {M}ons, {M}ons, 1980),
  vol.~33, 1981, pp.~131--155. \MR{MR620968 (82h:03022)}

\end{thebibliography}
\providecommand{\bysame}{\leavevmode\hbox to3em{\hrulefill}\thinspace}
\providecommand{\MR}{\relax\ifhmode\unskip\space\fi MR }
\providecommand{\MRhref}[2]{%
  \href{http://www.ams.org/mathscinet-getitem?mr=#1}{#2}
}
\providecommand{\href}[2]{#2}

\end{document}